\documentclass[11pt,a4paper]{article}       
\usepackage[affil-it]{authblk}
\usepackage{graphicx}
\usepackage{moreverb,url}

\newcommand\BibTeX{{\rmfamily B\kern-.05em \textsc{i\kern-.025em b}\kern-.08em
T\kern-.1667em\lower.7ex\hbox{E}\kern-.125emX}}

\usepackage{algorithm} 
\usepackage{algorithmic} 
\usepackage{amsthm}
\usepackage{amsmath} 
\usepackage{amssymb}  
\usepackage{a4wide}
\usepackage{array} 
\usepackage{bold-extra}
\usepackage{booktabs}
\usepackage{braket,amsfonts} 
\usepackage{capt-of}
\usepackage{epsfig} 
\usepackage{epstopdf} 
\usepackage{enumerate}
\usepackage{graphics} 
\usepackage{graphicx}
\usepackage{lscape}
\usepackage{mathptmx} 
\usepackage{mathtools}
\usepackage{mdwlist}
\usepackage{tikz}
  \usetikzlibrary{arrows,automata,decorations.pathmorphing,backgrounds,positioning,fit,petri}
  \usetikzlibrary{calc,intersections,through,backgrounds}
  \usetikzlibrary{plotmarks}
  \usetikzlibrary{trees}
\usepackage{times} 
\usepackage{pbsi}
\usepackage{picinpar}
\usepackage{psfrag}
\usepackage{pgfplots}
\usepackage{pstricks}
\usepackage{pst-node}
\usepackage{rotating}
\usepackage{subcaption}
\usepackage{xspace}
\usepackage{siunitx}
\usepackage[most]{tcolorbox}
\colorlet{texcscolor}{blue!50!black}
\colorlet{texemcolor}{red!70!black}
\colorlet{texpreamble}{red!70!black}
\colorlet{codebackground}{black!25!white!25}
\usepackage{url}
\usepackage{xcolor}

\newcommand{\bq}{\mathbf{q}}

\newcommand{\bd}{\mathbf{d}}
\newcommand{\bc}{\mathbf{c}}
\newcommand{\bg}{\mathbf{g}}
\newcommand{\btau}{\boldsymbol{\tau}}

\newcommand{\bmu}{\boldsymbol{\mu}}
\newcommand{\balpha}{\boldsymbol{\alpha}}
\newcommand{\bgamma}{\boldsymbol{\gamma}}

\newcommand{\bpsi}{\boldsymbol{\psi}}
\newcommand{\bD}{\boldsymbol{D}}
\newcommand{\bC}{\boldsymbol{C}}

\newcommand{\belll}{\boldsymbol{\ell}}
\newcommand{\Real}{\mathbb{R}}

\newtheorem{thrm}{Theorem}[section]
\newtheorem{lmm}[thrm]{Lemma}

\newtheorem{prpstn}[thrm]{Proposition}

\newtheorem{dfntn}[thrm]{Definition}

\newtheorem{rmrk}[thrm]{Remark}

\DeclareMathOperator{\diag}{diag}

\DeclareMathOperator{\glb}{\bigwedge}

\definecolor{mycolor1}{rgb}{0.00000,0.44700,0.74100}%
\definecolor{mycolor2}{rgb}{0.85000,0.32500,0.09800}%
\definecolor{mycolor3}{rgb}{0.92900,0.69400,0.12500}%
\definecolor{mycolor4}{rgb}{0.49400,0.18400,0.55600}%
\definecolor{mycolor5}{rgb}{0.46600,0.67400,0.18800}%

\begin{document}

\title{Graph-based algorithms for the efficient solution of a class of optimization problems}
\providecommand{\keywords}[1]{\textbf{\textit{Index terms---}} #1}
\author{Luca Consolini$^1$, Mattia Laurini$^1$, Marco Locatelli$^1$}

\date{\small $^1$ Dipartimento di Ingegneria e Architettura, Universit\`a degli Studi di Parma,\\ Parco Area delle Scienze 181/A, 43124 Parma, Italy.\\ luca.consolini@unipr.it, mattia.laurini@unipr.it, marco.locatelli@unipr.it}

\maketitle

\begin{abstract}
In this paper, we address a class of specially structured problems
that include speed planning, for mobile robots and robotic
manipulators, and dynamic programming. 
We develop two new numerical procedures,
that apply to the general case and to the
linear subcase. With numerical experiments, we show that the proposed
 algorithms outperform generic commercial solvers.
\end{abstract}
\keywords{Computational methods, Acceleration of convergence, Dynamic programming, Complete lattices}

\maketitle

\section{Introduction}
In this paper, we address a class of specially structured problems
of form
\begin{equation}
\label{eqn_prob_class}
\begin{aligned}
& \max_x f(x)\\
\textrm{subject to }\ & a \leq x \leq g(x) ,
\end{aligned}
\end{equation}
where $x \in \Real^n$, $a \in \Real^n$,
$f: \Real^n \rightarrow \Real$ is a continuous function, strictly monotone
increasing with respect to each component and
$g = (g_1, g_2, \ldots, g_n)^T: \Real^n \rightarrow \Real^n$,
is a continuous function such that, for $i = 1, \ldots, n$, $g_i$ is monotone not
decreasing with respect to all variables and constant with respect to
$x_i$. Also, we assume that there exists a real constant vector $U$ such that 
\begin{equation}
\label{eqn_for_f}
g(x) \leq U, \forall x: a \leq x \leq g(x)\,.
\end{equation}

A Problem related to~\eqref{eqn_prob_class} that is relevant
in applications is the following one
\begin{equation}
\label{eqn_prob_class_lin}
\begin{aligned}
& \max_x f(x)\\
\textrm{subject to }\ & 
0 \leq x \leq \underset{\ell \in \mathcal{L}}{\glb} \left\{ A_\ell x + b_\ell
  \right\},\
x \leq U,
\end{aligned}
\end{equation}
where, for each  $\ell \in \mathcal{L} = \{1, \ldots, L\}$, with $L \in \mathbb{N}$,
$A_\ell$ is a nonnegative matrix and $b_\ell$ is a nonnegative vector.

Note that the expression $\underset{\ell \in \mathcal{L}}{\glb}$,\ on the right hand
side of~\eqref{eqn_prob_class_lin}, denotes the greatest lower bound of $L$ vectors.
It corresponds to the component-wise minimum of vectors $A_\ell x + b_\ell$,
where a different value of $\ell \in \mathcal{L}$ can be chosen for
each component. We will show that Problem~\eqref{eqn_prob_class_lin}
is actually a subclass of~\eqref{eqn_prob_class} after a suitable
definition of function $g$ in~\eqref{eqn_prob_class}.

We will also show that the solution of Problems~\eqref{eqn_prob_class}
and~\eqref{eqn_prob_class_lin} is independent on the specific choice
of $f$. Hence, Problem~\eqref{eqn_prob_class_lin} is equivalent to the
following linear one
\begin{equation}
\label{eqn_prob_class_proglin}
\begin{aligned}
& \max_x \sum_{i=1}^n x_i\\
\textrm{subject to }\ &
0 \leq x , C x+ d \leq 0, x \leq U,
\end{aligned}
\end{equation}
where $C$ is a matrix such that every row contains one and only one
positive entry and $d$ is a nonpositive vector.

The structure of the paper is the following:
in Section~\ref{sec:appl} we justify the interest in Problem
class~\eqref{eqn_prob_class} and, in particular, its
subclass~\eqref{eqn_prob_class_lin}, by presenting some
problems in control, which can be reformulated as optimization
problems within subclass~\eqref{eqn_prob_class_lin}.
In Section~\ref{sec:prob_class} we derive some theoretical
results about Problem~\eqref{eqn_prob_class} and a class
of algorithms for its solution.
In Section~\ref{sec_lin_case} we do the same for the
subclass~\eqref{eqn_prob_class_lin}.
In Section~\ref{sec:convergence_speed_discussion} we discuss
some theoretical and practical issues about convergence speed
of the algorithms and we present some numerical experiments.
Some proofs are given in the appendix.

\subsection{Problems reducible to form~\eqref{eqn_prob_class_lin}}\label{sec:appl}
\subsubsection{Speed planning for autonomous vehicles}

\begin{figure}
\centering
\includegraphics[width = 2.5in]{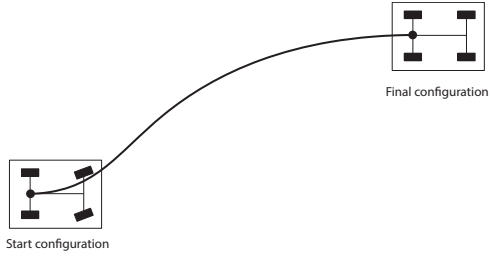}
\caption{A path to follow for an autonomous car-like vehicle.}
\label{fig:path-to-follow}
\end{figure}

This example is taken from~\cite{MinSCL17} and we refer the reader to
this reference for further detail.
We consider a speed planning problem for a mobile vehicle
(see Figure~\ref{fig:path-to-follow}).
We assume that the path that joins the initial and the final configuration
is assigned and we aim at finding the time-optimal speed law that satisfies
some kinematic and dynamic constraints.
Namely, we consider the following problem
\begin{subequations}
\label{eqn_problem_pr}
\begin{align}
& \min_{v \in C^1([0,s_f],\Real)} \int_0^{s_f} v^{-1}(s) d s \label{obj_fun_pr}\\
\textrm{subject to }\ &v(0)=0,\,v(s_f)=0 \label{inter_con_pr}\\
& 0< v(s) \leq  \bar v,& s \in (0,s_f), \label{con_speed_pr}\\
& |2 v'(s)v(s)| \leq A_T, &s \in [0,s_f],  \label{con_at_pr}\\
& |k(s)| v(s)^2 \leq A_N, &s \in [0,s_f],  \label{con_an_pr}
\end{align}
\end{subequations}
where $\bar v$, $A_T$, $A_N$ are upper bounds for the velocity, the
tangential acceleration and the normal acceleration, respectively.
Here, $s_f$ is the length of the path (that is assumed to be parameterized
according to its arc length) and $k$ is its scalar curvature (i.e., a function
whose absolute value is the inverse of the radius of the circle that locally
approximates the trajectory).

The objective function~\eqref{obj_fun_pr} is the total maneuver
time, constraints~\eqref{inter_con_pr} are the initial and final
interpolation
conditions and constraints~\eqref{con_speed_pr},~\eqref{con_at_pr},~\eqref{con_an_pr}
limit velocity and tangential and normal components of
acceleration.

After the change of variable $w=v^2$, the problem can be rewritten as

\begin{subequations}
\label{eqn_problem_cont}
\begin{align}
& \min_{w \in C^1([0,s_f],\Real)} \int_0^{s_f} w(s)^{-1/2} ds \label{obj_fun_cont}\\
\textrm{subject to }\ &w(0)=0,\,w(s_f)=0, \label{inter_con_cont}\\
& 0< w(s) \leq  \bar v^2, &s \in (0,s_f), \label{con_speed_cont}\\
& |w'(s)| \leq A_T, & s \in [0,s_f],  \label{con_at_cont}\\
& |k(s)| w(s) \leq A_N, & s \in [0,s_f].  \label{con_an_cont}
\end{align}
\end{subequations}

For $i=1\,\ldots,n$, set $w_i=w((i-1)h)$, with
$h=\frac{s_f}{n-1}$, then Problem~\eqref{eqn_problem_cont} can be
approximated with

\begin{subequations}
\label{eqn_problem_dis}
\begin{align}
& \min_{w\in \Real^n} \ \ \phi(w) \label{obj_fun_dis}\\
\textrm{subject to }\ & w_1=0,\,w_n=0, \label{inter_con_dis}\\
& 0< w_i \leq  \bar v^2, &i=2,\ldots,n-1, \label{con_speed_dis}\\
& |w_{i+1}-w_i| \leq h A_T, & i =1,\ldots,n-1,  \label{con_at_dis}\\
& |k(h(i -1))| w_i \leq A_N, & i = 2,\ldots,n-1,  \label{con_an_dis}
\end{align}
\end{subequations}
where the total time to travel the complete path is approximated by
\begin{equation}
\label{eqn_obj_discr}
\phi(w) = \sum_{i=1}^{n -1} t_i = 2h \sum_{i=1}^{n-1} \frac{1}{\sqrt{w_i}+ \sqrt{w_{i+1}}}.
\end{equation}

Note that conditions~\eqref{con_at_dis} is
obtained by Euler approximation of $w'(hi)$. Similarly,
the objective function~\eqref{eqn_obj_discr} is a discrete
approximation of the integral appearing in~\eqref{obj_fun_cont}.
By setting $f(w)=\phi(w)$, $a=0$, $g_1(w)=0$, $g_n(w)=0$ and,
for $i = 2, \ldots, n-1$,
\[
g_i(w)=\bigwedge \left\{\bar v^2,
\frac{A_N}{|k(h(i - 1))|}, h A_T + w_{i-1}, h A_T + w_{i+1}\right\}\,,
\]
Problem~\eqref{eqn_problem_dis} takes on the form of
Problem~\eqref{eqn_prob_class} and, since $g$ is linear with respect to
$w$, it also belongs to the more specific class~\eqref{eqn_prob_class_lin}.
We remark that, with respect to the problem
class~\eqref{eqn_prob_class_lin}, we minimize a decreasing function
which is equivalent to maximizing an increasing function.

Our previous works~\cite{Minari16},~\cite{MinSCL17} present an algorithm,
with linear-time computational complexity with respect to the number of variables
$n$, that provides an optimal solution of Problem~\eqref{eqn_problem_dis}.
This algorithm is a specialization of the algorithms proposed in this paper which
exploits some specific feature of Problem~\eqref{eqn_problem_dis}.
In particular, the key property of Problem~\eqref{eqn_problem_dis}, which
strongly simplifies its solution, is that functions $g_i$ fulfill the so-called
\emph{superiority condition} 
\[
g_i(w_{i-1},w_{i+1}) \geq w_{i-1}, w_{i+1},
\]
i.e., the value of function $g_i$ is not lower than each one of its arguments.

\subsubsection{Speed planning for robotic manipulators}
The technical details of this second example are more involved and we refer the reader
to~\cite{DBLP:journals/corr/abs-1802-03294}
for the complete discussion.
Let $\mathbb{R}^p$ be the
configuration space of a robotic manipulator with $p$-degrees of
freedom. The coordinate vector $\bq$ of a trajectory in $U$ satisfies the dynamic equation
\begin{equation}
\label{eq:manip}
\bD(\bq)\ddot{\bq} + \bC(\bq,\dot{\bq} )\dot{\bq} + \belll(\bq) = \btau,
\end{equation}
where $\bq \in \Real^{p}$ is the generalized position vector, $\btau \in \Real^{p}$ is the generalized force vector, $\bD(\bq)$ is the mass matrix, $\bC(\bq,\dot{\bq})$ is the matrix accounting for centrifugal and
Coriolis effects (assumed to be linear in $\dot{\bq}$) and $\belll(\bq)$ is the vector accounting for joints position dependent forces, including gravity.
Note that we do  not consider Coulomb friction forces.

Let $\bgamma \in C^2([0,s_{f}],\Real^{p}) $ be a function such that
($\forall \lambda \in [0,s_f]$)
$\lVert \bgamma^\prime(\lambda) \rVert =1$.
The image set  $\bgamma([0,s_f])$ represents the
coordinates of the elements of a reference path.
In particular, $\bgamma(0)$ and $\bgamma(s_f)$ are the coordinates of
the initial and final configurations. Define $t_{f}$ as the time when the robot reaches the end of the path. Let
 $\lambda : [0, t_f] \rightarrow [0, s_f]$ be a differentiable monotone increasing function that represents
 the position of the robot as a function of time and let   $ v : [0, s_f] \rightarrow [0, +\infty]$ be such that
  $\left( \forall t \in [0,t_f]\right) \dot{\lambda}(t) = v(\lambda(t))$. Namely, $v(s)$ is the velocity of the
 robot at position $s$. We impose ($\forall s \in [0,s_{f}]$) $v(s) \ge 0$.
For  any $t \in [0,t_f]$, using the chain rule, we obtain
\begin{equation}
\label{eq:rep}
\begin{array}{ll}
\bq(t) =& \bgamma(\lambda(t)),\\[8pt]
\dot{\bq}(t) = & \bgamma^{\prime}(\lambda(t))v(\lambda(t)),\\[8pt]
\ddot{\bq}(t) = & \bgamma^{\prime}(\lambda(t))v^\prime(\lambda(t))v(\lambda(t)) + \bgamma^{\prime\prime}(\lambda(t))v(\lambda(t))^2.
\end{array}
\end{equation}

Substituting (\ref{eq:rep}) into the dynamic equations (\ref{eq:manip}) and setting $s = \lambda(t)$, we rewrite the dynamic equation (\ref{eq:manip}) as follows:\\
\begin{equation}
\label{eq:dynamic}
\bd(s)v^{\prime}(s)v(s) + \bc(s)v(s)^2 + \bg(s) = \btau(s) ,
\end{equation}
where the parameters in (\ref{eq:dynamic}) are defined as
\begin{equation}
\label{eq:dynamic_parameters}
\begin{array}{l}
\bd(s) = \bD(\bgamma(s))\bgamma^{\prime}(s),\\ [8pt]
\bc(s) =  \bD(\bgamma(s))\bgamma^{\prime\prime}(s)  + \bC(\bgamma(s),\bgamma^{\prime}(s))\bgamma^{\prime}(s),  \\ [8pt]
\bg(s) = \belll(\bgamma(s)).
\end{array}
\end{equation}
The objective function is given by the overall travel time $t_f$
defined as
\begin{equation}
\label{eq:objective}
\displaystyle t_f = \int_0^{t_f}1\,dt = \int_{0}^{s_f} v(s)^{-1}\, ds.
\end{equation}

Let $\bmu, \bpsi, \balpha : \left[ 0, s_f \right] \rightarrow \Real^{p}_{+}$ be assigned bounded functions and
consider the following minimum time problem:
\begin{subequations}\label{prob:1}
\begin{align}
& \displaystyle\min_{v \in C^{1},\btau\in C^{0}} \displaystyle\int_0^{s_f} v(s)^{-1} \, ds, \label{obj:v}\\
\textrm{subject to }\ &  \ (\forall s \in [0,s_{f}]) \nonumber\\
& \bd(s)v^{\prime}(s)v(s) + \bc(s)v(s)^2 + \bg(s) = \btau(s), \label{con:dynamic}\\
& \bgamma^{\prime}(s)v(s) = \dot{\bq}(s),\label{con:kinematic1} \\
&   \bgamma^{\prime}(s)v^\prime(s)v(s) + \bgamma^{\prime\prime}(s) v(s)^{2} = \ddot{\bq}(s),\label{con:kinematic2} \\
& \lvert  \btau(s) \rvert  \le \bmu(s), \label{con:force_bound}\\
&  \lvert   \dot{\bq}(s) \rvert \le \bpsi(s),\label{con:vel_bound} \\
& \lvert \ddot{\bq}(s) \rvert \le \balpha (s), \label{con:acc_bound}\\
&v(s) \ge 0, \label{con:positive-velocity} \\
& v(0) = 0, \, v(s_f) =  0,  \label{con:interpolation}
\end{align}
\end{subequations}
where (\ref{con:dynamic}) represents the robot dynamics,
(\ref{con:kinematic1})-(\ref{con:kinematic2}) represent the
relation between the path $\bgamma$ and the generalized
position $\bq$ shown in~(\ref{eq:rep}), (\ref{con:force_bound})
represents the bounds on generalized forces,
(\ref{con:vel_bound}) and (\ref{con:acc_bound}) represent the
bounds on joints velocity and acceleration, respectively.
Constraints~(\ref{con:interpolation}) specify the interpolation
conditions at the beginning and at the end of the path.

After some manipulation and using a carefully chosen finite
dimensional approximation (again,
see~\cite{DBLP:journals/corr/abs-1802-03294} for the details),
Problem~\eqref{prob:1} can be reduced to form (see
Proposition~8 of~\cite{DBLP:journals/corr/abs-1802-03294}).

\begin{equation}
\label{eq:probl}
\begin{aligned}
& \min_w \phi(w) \\
\textrm{subject to }\ & w_i\leq f_{j,i} w_{i+1} + c_{j,i} & i=1,\ldots,n-1, \quad j=1,\ldots,p, \\[5pt]
& w_{i+1}\leq b_{k,i} w_{i} + d_{k,i} & i=1,\ldots,n-1, \quad k=1,\ldots,p,\\[5pt]
& 0\leq w_i\leq u_i & i=1,\ldots,n,
\end{aligned}
\end{equation}
where, $\phi$ is defined as in~\eqref{eqn_obj_discr} and
$w = (w_1, \ldots, w_n)^T$.
For $i = 1, \ldots, n$, $w_i = v((i-1)h)^2$, $h = \frac{s_f}{n-1}$, is the
squared manipulator speed at configuration $\bgamma((i-1)h)$.
Moreover $u_i$, $f_{j,i}$, $c_{j,i}$, $b_{k,i}$, $d_{k,i}$ are nonnegative
constant terms depending on problem data.

Problem~\eqref{eq:probl} belongs to classes~\eqref{eqn_prob_class}
and~\eqref{eqn_prob_class_lin}.
Also in this case,  the performance of the algorithms proposed in this
paper can be enhanced by exploiting some further specific features of
Problem~\eqref{eq:probl}. In particular,
in~\cite{DBLP:journals/corr/abs-1802-03294}, we were able to develop a
version of the algorithm with optimal time-complexity $O(n p)$.

\subsubsection{Dynamic Programming} \label{sec:motivation}

This section is based on Appendix~A of~\cite{bardi2008optimal},
to which we refer the reader for more detail.
Consider a control system defined by the following differential
equation in $\Real^n$:
\begin{equation} \label{eq:controlProblem}
\begin{cases}
\dot{x}(t) = f(x(t),u(t)) \\
x(0) = x_{0},
\end{cases}
\end{equation}

\noindent
where $f:\Real^n \times U \rightarrow \Real^n$ is a continuous
function, $x_0$ is the initial state, $u(t) \in U \subset \Real^m$ is
the control input and $U$ is a compact set of admissible controls.
Consider an infinite horizon cost functional defined as follows
\begin{equation}\label{eq:costFunc}
J_{x_0} (u) = \int\limits_0^\infty g(x(t), u(t)) e^{-\lambda t} dt,
\end{equation}

\noindent
where $g:\Real^n \times U  \rightarrow  \Real$ is a continuous
cost function.
The viscosity parameter $\lambda$ is a positive real constant.
Following~\cite{bardi2008optimal}, we assume that there exist
positive real constants
$L_f$, $L_g$, $C_f$, $C_g$ such that,
$\forall x_1, x_2 \in \Real^n$, $\forall u \in U$,
\begin{align*}
| f(x_1, u) - f(x_2, u) | \leq L_f | x_1 - x_2 |,\qquad &
\left\|f(x_1, u)\right\|_\infty \leq C_f,\\
| g(x_1, u) - g(x_2, u) | \leq L_g | x_1 - x_2 |,\qquad &
\left\|g(x_1, u)\right\|_\infty \leq C_g.
\end{align*}

Define the value function $v: \Real^n \to \Real$ as
\[
v(x_0) = \inf_{u \in U} J_{x_0}(u).
\]

\noindent
As shown in \cite{bardi2008optimal}, the value function $v$ is
the unique viscosity solution of the Hamilton-Jacobi-Bellman (HJB) equation:
\begin{equation}\label{eq:HJ}
\lambda v(x) + \sup_{u \in U}\{- \nabla v(x) f(x, u) - g(x, u)\} = 0,\quad
x \in \Real^n,
\end{equation}

\noindent
where $\nabla v$ denotes the gradient of $v$.

In general, a closed form solution of the partial differential
equation~\eqref{eq:HJ} does not exist.
Various numerical procedures have been developed to compute
approximate solutions, such as in~\cite{4554208}, \cite{bardi2008optimal}
\cite{6328288}, \cite{Wang01062000}.
 
In particular,~\cite{bardi2008optimal} presents an approximation
scheme based on a finite approximation of state and control
spaces and a discretization in time.
Roughly speaking, in~\eqref{eq:HJ} one can approximate
$\nabla v(x) f(x,u) \simeq h^{-1} (v(x + h f(x,u)) - v(x))$,
where $h$ is a small positive real number that represents an
integration time.
In this way,~\eqref{eq:HJ} becomes
\begin{equation*}
(1 + \lambda h) v(x) = \min_{u \in U}\{ v(x + h f(x,u)) + h g(x,u) \} = 0,
\ x \in \Real^n,
\end{equation*}

\noindent
and, by approximating  $(1 + \lambda h)^{-1} \simeq (1 - \lambda h)$,
$(1 + \lambda h)^{-1} h \simeq h$, one arrives at the following
HJB equation in discrete time
\begin{equation}\label{eq:HJtime}
v_{h}(x) = \min_{u \in U}\left\{(1-\lambda h)v_{h}(x+hf(x,u)) + hg(x,u) \right\},\ x \in \Real^n.
\end{equation}

\noindent
For a more rigorous derivation of~\eqref{eq:HJtime},
again, see~\cite{bardi2008optimal}.

A triangulation is computed on a finite set of vertices
$\mathcal{T} = \{x_i\}_{i \in \mathcal{V}} \subset \Real^n$, with
$\mathcal{V} \subseteq \mathbb{N}$ and $|\mathcal{V}| = N$.
Evaluating~\eqref{eq:HJtime} at $x \in \mathcal{T}$, we obtain
\begin{align} \label{eq:HJtimespace}
 v_h(x_i) = \min_{u \in U}\left\{(1 - \lambda h) v_h(x_i + hf(x_i, u)) + hg(x_i, u) \right\},
 \ i \in \mathcal{V}. &
\end{align}

\noindent
Note the dependence of the value cost function on the choice of the integration step $h$.
Using the triangulation, function $v$ can be approximated by a linear
affine function of the finite
set of variables $v_h(x_i)$, with $i \in \{1, \ldots, N\}$.

Theorem~2.1 of Appendix~A of~\cite{bardi2008optimal} shows that, if $\lambda > L_f$ and $h \in \left(\left. 0,\frac{1}{\lambda}\right]\right.$, system~\eqref{eq:HJtimespace} has a unique solution that converges
uniformly to the solution of~\eqref{eq:HJ} as $h,d,\frac{d}{h}$ tend
to $0$, where $d$ is the maximum diameter of the simplices used in the triangulation.
Note that, for convergence results, one should choose $\lambda$ large enough since it is
bounded from below by $L_f$.

To further simplify~\eqref{eq:HJtimespace}, it is possible to discretize the control space,
substituting $U$ with a finite set  of controls $\{u_\ell\}_{\ell\ \in \mathcal{L}}$, so that we can
replace~\eqref{eq:HJtimespace} with
\begin{equation} \label{eq:HJtimespacecontols}
\begin{aligned}
 v_h(x_i) = \min_{\ell \in \mathcal{L}}\left\{(1-\lambda h) v_h(x_i + hf(x_i, u_\ell)) + hg(x_i, u_\ell) \right\},
 \ i \in \mathcal{V}.
\end{aligned}
\end{equation}

Figure~\ref{fig:triang} illustrates a step of construction of problem~\eqref{eq:HJtimespacecontols}.
Namely, for each node of the triangulation $x_i$ and each value of the control $u_\ell$, all end points
$x_i + h f(x_i,u_\ell)$ of the Euler approximation of the solution of~\eqref{eq:controlProblem} from the
initial state $x_i$ are computed. The value cost function for these end points is given by a convex
combination of its values on the triangulation vertices.

\begin{figure}[!ht]
  \centering
   \includegraphics[width=0.5\columnwidth]{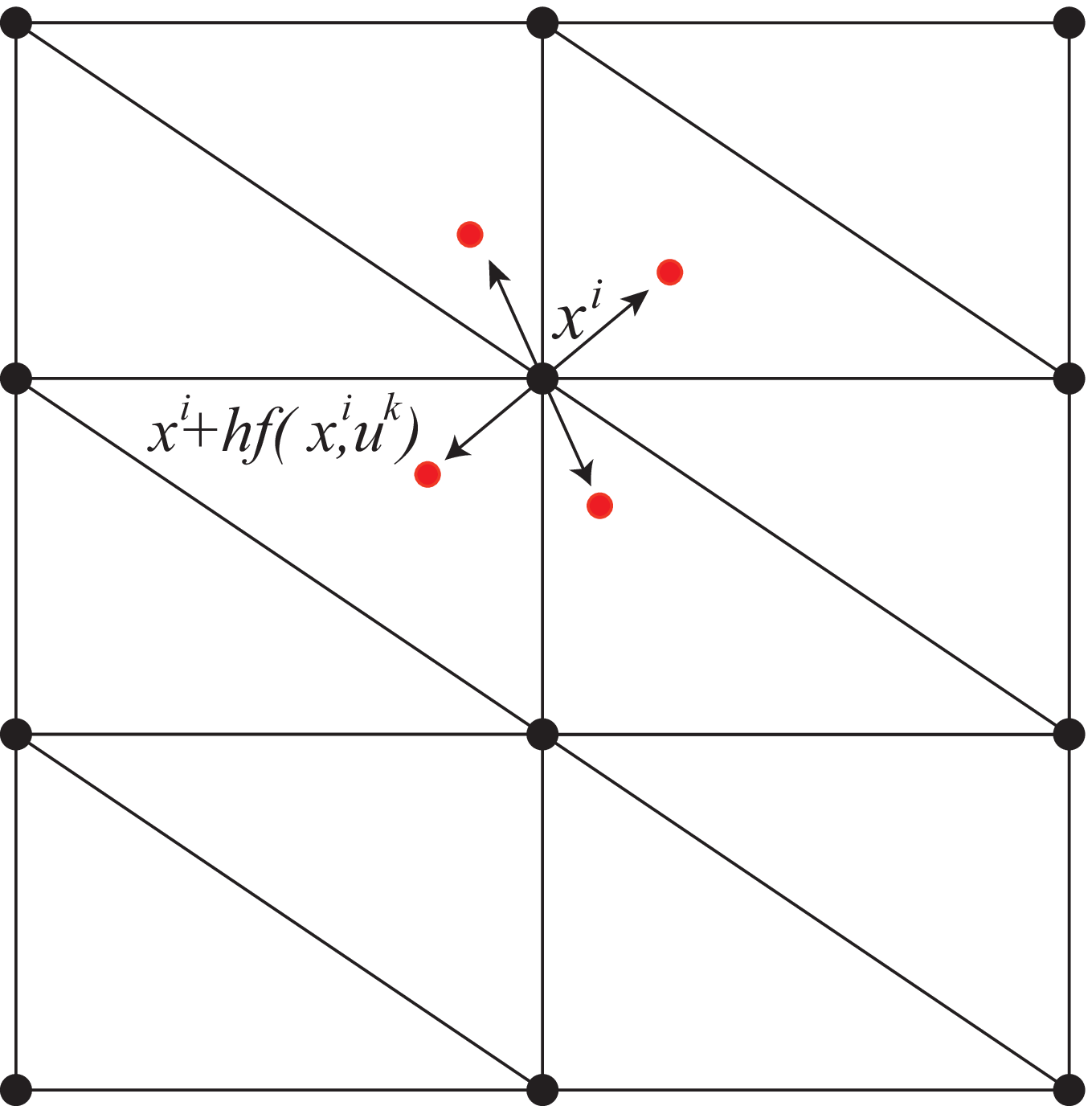}
  \caption{Approximation of the HJB equation on a triangulation with four controls.}
  \label{fig:triang}
\end{figure}

Set vector $w:= [w_1,\ldots,v_n]^T={[v_h(x_1), v_h(x_2), \ldots, v_h(x_N)]}^T$,
in this way $w \in \Real^N$ represents
the value of the cost function on the grid points.

Note that, for each $x_i,u_\ell$, the right-hand side of~\eqref{eq:HJtimespacecontols} is affine with
respect to $w$, so that Problem~\eqref{eq:HJtimespacecontols}
can be rewritten in form
\[
\begin{aligned}
& \max_w \sum_{i} w_i\\
\textrm{subject to }\ & 
0 \leq w \leq \underset{\ell \in \mathcal{L}}{\glb} \left\{ A_\ell w + b_\ell
  \right\},\
w \leq \frac{1}{\lambda}\,,
\end{aligned}
\]
where for $\ell \in \mathcal{L}$,
$A_\ell \in \Real^{N \times N}$ are suitable nonnegative matrices and
$b_\ell \in \Real^{N}$ are suitable nonnegative vectors. Hence,
Problem~\eqref{eq:HJtimespacecontols} belongs to class~\eqref{eqn_prob_class_lin}.
Moreover, observe that if $h$ is sufficiently small, matrices
${\left\{ A_\ell \right\}}_{\ell \in \mathcal{L}}$ are dominant diagonal.

\subsection{Statement of Contribution}
The main contributions of the paper are the following ones:

\begin{itemize}
\item We develop a new procedure (Algorithm~\ref{alg:eps_sol}) for the solution of
  Problem~\eqref{eqn_prob_class} and a more specific one
  (Algorithm~\ref{alg:consensus}) for its
  subclass~\eqref{eqn_prob_class_lin}. We prove the correctness of
  these solution methods.
\item With numerical experiments, we show that the proposed
  algorithms outperform generic commercial solvers in the solution of linear problem~\eqref{eqn_prob_class_lin}.
\end{itemize}

\subsection{Notation}\label{subsec:notation}

The set of nonnegative real numbers is denoted by
$\Real_{+} := [0, +\infty)$ and $\underline{0}$ denotes the
zero vector of $\Real^n$.

Given $n, m \in \mathbb{N}$, let $x \in \Real^n$ and
$A \in \Real^{n \times m}$, for $i \in \{ 1, \ldots, n \}$, we denote the
$i$-th component of $x$ with ${[x]}_i$ and the $i$-th row of $A$
with ${[A]}_{i*}$; further, for $j \in \{ 1, \ldots, m \}$ we denote
the $j$-th column of $A$ with ${[A]}_{*j}$ and the $ij$-th element
of $A$ with ${[A]}_{ij}$.

Function $\left\| \cdot \right\|_\infty: \Real^n \rightarrow \Real_+$
is the infinity norm, namely the maximum norm, of $\Real^n$ (i.e.,
$\forall \ x \in \Real^n \ \left\| x \right\|_\infty = \max\limits_{i \in \{1, \ldots, n\}}{| {[x]}_i |}$);
$\left\| \cdot \right\|_\infty$ is also used to denote the induced matrix norm.
Given a finite set $S$, the cardinality of $S$ is denoted by $|S|$, the power
set of $S$ is denoted by $\wp(S)$ and symbol $\varnothing$ denotes the
empty set.

Consider the binary relation $\leq$ defined on $\Real^n$ as follows
\[
\forall x, y \in \Real^{n}\
(x \leq y\ \Longleftrightarrow\ y - x \in \Real_+^n).
\]
It is easy to verify that $\leq$ is a \textit{partial order} of
$\Real^n$.

Finally, given a nonempty set $\mathcal{V}$ let us define a priority
queue $Q$ as a finite subset of
$\mathcal{Q} := \mathcal{V} \times \mathbb{R}$ such that,
if $(v, q) \in Q$, then, no other element $(\bar v, \bar q) \in Q$ can
satisfy that $\bar v = v$.
Let us also define two operations on priority queues:
$\text{Enqueue}: \wp(\mathcal{Q}) \times \mathcal{Q} \rightarrow \mathcal{Q}$,
which, given $Q \in \wp(\mathcal{Q})$ and $(v, q) \in \mathcal{Q}$,
if $Q$ does not contain elements of the form $(v, p)$,
with $p \geq q$, then $\text{Enqueue}$ adds $(v, q)$ to the priority queue
$Q$ and removes any other element of the form $(v, p)$, with $p < q$,
if previously present.
The second operation we need on priority queues is
$\text{Dequeue}: \wp(\mathcal{Q}) \rightarrow \wp(\mathcal{Q}) \times \mathcal{V}$
which extracts from a priority queue $Q$ the pair $(v, q)$ with highest priority
(i.e., it extracts $(v, q) \in Q$ such that $\forall (\bar v, \bar q) \in Q, q \geq \bar q$)
and returns element $v$.

\section{Characterization of Problem~\eqref{eqn_prob_class}}
\label{sec:prob_class}

In this section, we consider Problem~\eqref{eqn_prob_class} 
with the additional assumption $g(a) \geq a$ which guarantees
that the feasible set of Problem~\eqref{eqn_prob_class} 
\[
\Sigma=\{x \in \Real^n: a \leq x \leq g(x)\}\,
\]
is non-empty.

For any $\Gamma \subset \Sigma$ define $\bigvee \Gamma$
as the smallest $x \in \Sigma$, if it exists, such that
$(\forall y \in \Gamma)\, x \geq y$.
We call $\bigvee \Gamma$ the \emph{least upper bound} of $\Gamma$.
Note that $\bigvee \varnothing = a$.
The following proposition shows that $\bigvee \Gamma$ exists.
\begin{prpstn}
\label{prop_closure}
For any $\Gamma \subset \Sigma$, $\bigvee \Gamma$ exists.
\end{prpstn}
\begin{proof}

We first prove that, if $x,y \in \Sigma$, then $x \vee y \in \Sigma$
(recall that $\vee$ denotes the component-wise maximum).
It is obvious that $x \vee y \geq a$. Thus, we only need to prove
that, for each $j=1,\ldots,n$, $[y \vee x]_j \leq g_j(x \vee y)$.
To see this,
let us assume, w. l. o. g. , that $[x]_j \leq [y]_j$. Since $y \in
\Sigma$, then $[y]_j \leq g_j(y)$. Moreover, $g_j(y) \leq g_j(y \vee
x)$ since $g_j$ is monotone
non decreasing, so that $[y \vee x]_j \leq g_j(y \vee x)$ as we wanted
to prove. 

Set $\Sigma$ is closed since it is defined by non strict
inequalities of a continuous function, $\Sigma$ is bounded by
assumption,
hence $\Sigma$ is compact.
Set $x^+=\bigvee \Sigma$,
note that $x^+ \leq U$ since $(\forall x \in \Sigma) x \leq
U$, where $U$ is defined in~\eqref{eqn_for_f}.
There exists a sequence $x: \mathbb{N} \to
\Sigma$
such that $\lim_{k \to \infty} x(k)=x^+$. Namely, for any
$k >0$, choose $x^{(1)}_k,\ldots,x^{(n)}_k \in \Sigma$ such that
$[x^+-x^{(i)}_k]_i< k^{-1}$ and set $x(k) =\bigvee \{x^{(1)}_k,\ldots,x^{(n)}_k\}$.
Being $\Sigma$ compact, $\Sigma$ is
also sequentially compact and $x^+ \in \Sigma$.
\end{proof}

Similarly, define $\bigwedge \Gamma$ as the largest $x$, if it exists,
such that $(\forall y \in \Gamma) x \leq y$, we call $\bigwedge
\Gamma$ the \emph{greatest lower bound} of $\Gamma$.

For $x,y \in \Sigma$, note that $x \vee y = \bigvee \{x,y\}$,
$x \wedge y = \bigwedge \{x,y\}$.

The following proposition characterizes set $\Sigma$ with
respect to operations $\vee$, $\wedge$. In particular, it shows that
the component-wise minimum and maximum of each subset of $\Sigma$
belongs to $\Sigma$.

\begin{prpstn}
Set $\Sigma$ with operations $\vee,\wedge$ defined above is a
complete lattice.
\end{prpstn}

\begin{proof}
It is a consequence of the dual of Theorem~2.31
of~\cite{davey2002introduction}.
Indeed $\Sigma$ has a bottom element ($a$) and $\bigvee \Gamma$
exists for any non-empty  $\Gamma \subset \Sigma$ by
Proposition~\ref{prop_closure}.
\end{proof}

A consequence of the previous definition is that also $\bigwedge \Gamma$ exists.

The following proposition shows that the least upper bound $x^+$ of $\Sigma$
is a fixed point of $g$ and corresponds to an optimal solution of
Problem~\eqref{eqn_prob_class}.

\begin{prpstn}
\label{prop_max_x}
Set
\[
x^+=\bigvee \Sigma\,,
\]
then 
i)
\begin{equation}
\label{eqn_fix_point}
x^+=g(x^+)
\end{equation}

ii) $x^+$ is an optimal solution of problem~\eqref{eqn_prob_class}.
\end{prpstn}

\begin{proof}
i) It is a consequence of Knaster-Tarski Theorem (see Theorem~2.35
of~\cite{davey2002introduction}), since $(\Sigma,\wedge,\vee)$ is a
complete lattice and $g$ is an order-preserving map.

ii) By contradiction, assume that $x^+$ is not optimal, this implies that
there exists $x \in \Sigma$ such that $f(x) > f(x^+)$. Being $f$
monotonic increasing, this implies that there exists $i \in
1,\ldots,n$ such that $[x]_i > [x^+]_i$, which implies that $x^+ \neq
\bigvee \Sigma$.
\end{proof}

\begin{rmrk}
The previous proposition shows that the actual form of function $f$ is
immaterial to the solution of Problem~\eqref{eqn_prob_class}, since
the optimal solution is $x^+$ for any strictly monotonic increasing
objective function $f$.
\end{rmrk}

The following defines a relaxed solution of
Problem~\eqref{eqn_prob_class}, obtained by allowing an error
on fixed-point condition~\eqref{eqn_fix_point}.

\begin{dfntn}
Let $\epsilon$ be a positive real constant, $x$ is
an $\epsilon$-solution of~\eqref{eqn_prob_class} if 
\[
x \geq a, \quad
\|x -g(x)\|_\infty < \epsilon\,.
\]
\end{dfntn}

The following proposition presents a sufficient condition that
guarantees that a sequence of $\epsilon$-solutions approaches $x^+$ as $\epsilon$ converges to $0$.

\begin{prpstn}
If there exists $\delta>0$ such that
\begin{equation}
\label{eqn_cond_xy}
(\forall x,y \geq a) \, \frac{\|g(x)-g(y)\|_\infty}{\|x-y\|_\infty} \notin [1-\delta,1+\delta] 
\end{equation}
then, there exists a constant $M$ such that, for any $\epsilon>0$,
if $x \in \Real^n$ is an $\epsilon$-solution of~\eqref{eqn_prob_class},
then 
\[
\|x-x^+\|_\infty \leq M \epsilon\,.
\]
\end{prpstn}

\begin{proof}
Let  $x$ be an $\epsilon$-solution.  By Proposition~\ref{prop_max_x}
we have that
\[
x-x^+= g(x) -g(x^+)  +\xi\,,
\]
where $\|\xi\|_\infty \leq \epsilon$. 

By assumption~\eqref{eqn_cond_xy}, either 
$\|g(x) - g(y)\|_\infty > (1 + \delta) \|x - y\|_\infty$ or
$\|g(x) - g(y)\|_\infty < (1 - \delta) \|x - y\|_\infty$.
In the first case, 
\[
\|x-x^+\|_{\infty}  \geq -\|\xi\|_{\infty} + (1+\delta) \|x-x^+\|_{\infty}\,,
\]
in the second case,
\[
\|x-x^+\|_{\infty}  \leq \|\xi\|_{\infty} + (1-\delta) \|x-x^+\|_{\infty}\,.
\]
In both cases it follows that
\[
 \|x-x^+\|_{\infty} \leq \delta^{-1} \|\xi\|_{\infty} \leq \delta^{-1} \epsilon\,.
\]
\end{proof}

\begin{rmrk}
If condition~\eqref{eqn_cond_xy} is not satisfied, an
$\epsilon$-solution of~\eqref{eqn_prob_class} can be very distant from
the optimal solution $x^+$.
Figure~\ref{fig:hypothesis} refers to a simple instance of
Problem~\eqref{eqn_prob_class} with $x \in \Real$, so that $g$ is a scalar
function. The optimal value $x^+$ corresponds to the maximum value of
$x$ such that $x\leq g(x)$. The figure also shows $\tilde x$, which is
an $\epsilon$-solution, for the value of $\epsilon$ depicted in the
figure. In this case there is a large separation between $x^+$ and
$\tilde x$. Note that in this case function $g$ does not
satisfy~\eqref{eqn_cond_xy}.
\end{rmrk}

\begin{figure}[h!]
\centering

\psfrag{x}{$x$}
\psfrag{g}{$g(x)$}
\psfrag{t}{$x$}
\psfrag{e}{$\epsilon$}
\psfrag{w}{$x^+$}
\psfrag{y}{$\tilde x$}
\includegraphics[width=.6\textwidth]{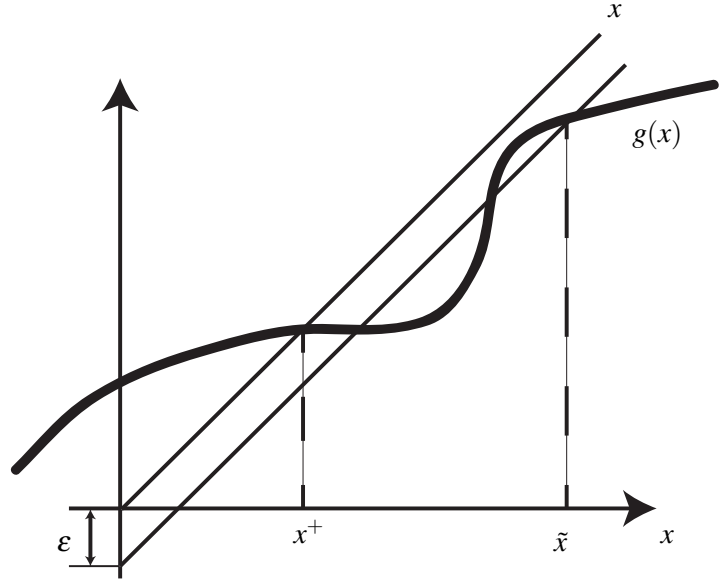}
\caption{Representation of an instance of problem~\eqref{eqn_prob_class} in which conditions~\eqref{eqn_cond_xy} does
  not hold.}
\label{fig:hypothesis}
\end{figure}

\begin{rmrk}\label{remark:fixed_point}
If $g$ is a contraction, namely, if there exists $\gamma \in [0, 1)$,
such that $\forall x, y \in \Real^n\ \| g(x) - g(y) \|_\infty \leq \gamma \|
x - y \|_\infty$ (a subcase of~\eqref{eqn_cond_xy}), then $x^+$
can be found with a standard fixed point iteration

\begin{equation}
\label{PROB}
\begin{cases}
x(k+1) = g(x)\\
x(0) = x_0,
\end{cases}
\end{equation}
and, given $\epsilon > 0$,
an $\epsilon$-solution $x$ of~\eqref{eqn_prob_class} can be computed with
Algorithm~\ref{alg:fixed_point}.
This algorithm, given an input tolerance $\epsilon$, function $g$
and an initial solution $x_0 \in \Real^n$, repeats the fixed point
iteration $x=g(x)$ until $x$ satisfies the definition of
$\epsilon$-solution, that is, until the infinity norm of error vector $\xi=x-g(x)$ is
smaller than the assigned tolerance $\epsilon$.
\end{rmrk}

\begin{algorithm}[h!]
\caption{Fixed Point Iteration.}
\label{alg:fixed_point}
\begin{algorithmic}[1]
\STATE INPUT: initial vector $x_0$, tolerance $\epsilon$,
function $g$.
\STATE OUTPUT: vector $x$.

\STATE

\STATE $x := x_0$

\REPEAT
    \STATE $x_{\textrm{old}} := x$
    \STATE $x := g(x)$ \label{step_fix_point}
    \STATE $\xi := x_{\textrm{old}}-x$
\UNTIL{$\|\xi\|_\infty \leq \epsilon$}
\STATE\RETURN{$x$}
\end{algorithmic}
\end{algorithm}

The special structure of Problem~\eqref{eqn_prob_class} leads to a
solution
algorithm that is much more efficient than
Algorithm~\ref{alg:fixed_point} in terms of overall number of elementary operations.
As a first step, we associate a graph to constraint $g$ of
Problem~\eqref{eqn_prob_class}.

\subsection{Graph associated to Problem~\eqref{eqn_prob_class}}

It is natural to associate to Problem~\eqref{eqn_prob_class} a directed graph
$\mathbb{G} = (V,E)$, where the nodes correspond to the $n$ components
of $x$ and of constraint $g$,
namely $V=\mathcal{V} \cup \mathcal{C}$, with
$\mathcal{V}=\{v_1,\ldots,v_n\}$, $\mathcal{C}=\{c_1,\ldots,c_n\}$,
where $v_i$ is the node associated to $[x]_i$ and
$c_i$ is the node associated to $g_i$.
The edge set $E \subseteq V \times V$ is defined according to the rules:
\begin{itemize}
\item for $i=1,\ldots,n$, there is a directed edge from $c_i$ to $v_i$,
\item for $i=1,\ldots,n$, $j=1,\ldots,n$, there is a directed edge
  from $v_i$ to $c_j$ if $g_j$ depends on $x_i$,
\item no other edges are present in $E$.
\end{itemize}

For instance, for $x \in \mathbb{R}^3$ consider problem
\[
\begin{aligned}
& \max_x f(x)&\\
\textrm{subject to }\ & 
0 \leq x_1 \leq g_1(x_2,x_3) \\
&0 \leq x_2 \leq g_2(x_1) \\
&0 \leq x_3 \leq g_3(x_1,x_2). 
\end{aligned}
\]

The associated graph, with
$\mathcal{V}=\{v_1,v_2,v_3\}$, $\mathcal{C}=\{c_1,c_2,c_3\}$, is given by

\begin{center}
\begin{tikzpicture}[->,>=stealth',shorten >=1pt,auto,node distance=2cm]
\tikzstyle{every state}=[fill=none,draw=black,text=black, minimum size = 29pt]

\node[state] (V1) {$v_1$};
\node[state] (V2) [right of=V1] {$v_2$};
\node[state] (V3) [right of=V2] {$v_3$};
\node[state] (C1) [above of=V1] {$c_1$};
\node[state] (C2) [right of=C1] {$c_2$};
\node[state] (C3) [right of=C2] {$c_3$};

\path (C1) edge (V1)
         (C2) edge (V2)
         (C3) edge (V3)
         (V2) edge (C1)
         (V3) edge (C1)
         (V1) edge (C2)
         (V1) edge (C3)
         (V2) edge (C3);
\end{tikzpicture}
\end{center}

We define the set of neighbors of node $i \in \mathcal{V}$ as
\[
\mathcal{N}(i) := \left\{ j \in \mathcal{V}\ |\ \exists c \in \mathcal{C} : (i, c), (c, j) \in E \right\},
\]
namely, a node $j \in \mathcal{V}$ is a neighbor of $i$ if there
exists a directed path of length two that connects $i$ to $j$.
For instance, in the previous example, $v_1 \in \mathcal{N}(v_3)$ and
 $v_2 \notin \mathcal{N}(v_3)$.
In other words, $v_j \in \mathcal{N}(v_i)$ if constraint $g_j$ depends
on $x_i$.

\subsection{Selective update algorithm for
  Problem~\eqref{eqn_prob_class}}

In Algorithm~\ref{alg:fixed_point}, each time
line~\ref{step_fix_point} is evaluated, the value of
all components of $x$ is updated according to the fixed point
iteration $x=g(x)$, even though many of them may
remain unchanged.  
We now present a more efficient procedure for computing an
$\epsilon$-solution of~\eqref{eqn_prob_class}, in which we update only the
value of those components of $x$ that are known to undergo a variation.
The algorithm is composed of two phases, an initialization and a main loop.
In the \emph{initialization}, $x$ is set to an initial value $x_0$ that is
known to satisfy $x_0 \geq x^+$. Then the fixed point error $\xi=x-g(x)$
is computed and all indexes $i=1,\ldots,n$ for which $[\xi]_i >
\epsilon$ are inserted into a priority queue, ordered with respect to a
policy that will be discussed later. In this way, at the end of the
initialization, the priority queue contains all indexes $i$ for which
the corresponding fixed point error $[\xi]_i$ exceeds $\epsilon$.

Then, the \emph{main loop} is repeated until the priority queue is
empty. First, we extract from the priority queue the index $i$ with the highest priority. Then, we update its value by setting
$[x]_i=g_i(x)$ and update the fixed point error $\xi$
by setting $[\xi]_j=[x]_j-g_j(x)$
for all
variables $j \in \mathcal{N}(i)$. This step is actually the key-point
of the algorithm: we recompute the fixed point error \emph{only} of
those variables that correspond to components of $g$ that we know to
have been affected by the change in variable $[x]_i$.
Finally, as in the initialization, all variables $j \in \mathcal{N}(i)$
such that the updated fixed-point error satisfies $[\xi]_j > \epsilon$
are placed into the priority queue.

The order in which nodes are actually processed depends on the
ordering of the priority
queue. The choice of this ordering turns out to be critical in terms of
computational cost for the algorithm, as can be seen in the numerical
experiments in Section~\ref{subsec:simulations}.
Various orderings for the priority queue will be introduced in
Section~\ref{subsec:simulations} and the ordering choice will be
discussed in more detail. The procedure stops once the priority queue
becomes empty, that is, once none of the updated nodes undergoes a
significant variation.  As we will
show, the correctness of the algorithm is independent on the choice of
the ordering of the priority queue.

We may think of graph $\mathbb{G}$ as a communication network in which
each node transmits its updated value to
its neighbours, whilst all other nodes maintain their value
unchanged.

These considerations lead to Algorithm~\ref{alg:eps_sol}.
This algorithm takes as input an initial vector $x_0 \in \Real^n$,
a tolerance $\epsilon$, function $g$ and the lower bound $a$.
From lines~\ref{alg2_initialization_begin} to~\ref{alg2_initialization_end}
it initializes the solution vector $x$, the priority queue $Q$ and the error
vector $\xi$.
From line~\ref{alg_first_iter_begin} to~\ref{alg_first_iter_end} it adds
into the priority queue those component nodes whose corresponding
component of the error vector $\xi$ is greater than tolerance $\epsilon$.
The priority with which a node is added to the queue will be
discussed in Section~\ref{subsec:simulations}, here symbol * denotes a
generic choice of priority.
Lines from~\ref{alg_main_loop_begin} to~\ref{alg_main_loop_end}
constitute the main loop.
While the queue is not empty, the component node $i$ with highest
priority is extracted from the queue and its value is updated.
Then, each component node $j$ which is a neighbor of $i$ is examined;
the variation of node $j$ is updated and,
if it is greater than tolerance $\epsilon$, neighbor $j$ is added to the
priority queue.
After this, the component corresponding to node $i$ in $\xi$ is set to 0.
Finally, once the queue becomes empty, the feasibility of
solution $x$ is checked and returned along with vector $x$.
We remark that Algorithm~\ref{alg:eps_sol} can be seen as a
generalization of Algorithm~\ref{alg:fixed_point} in~\cite{Cabassi2018},
where a specific priority queue (namely, one based on the values of the
nodes) was employed.
Also note that Algorithm~\ref{alg:eps_sol} can be seen as a bound-tightening
technique (see, e. g., \cite{Belotti2009}) which, however, for this specific class
of problem is able to return the optimal solution.

\begin{algorithm}[h!]
\caption{Solution algorithm for Problem~\eqref{eqn_prob_class}}
\label{alg:eps_sol}
\begin{algorithmic}[1]
\STATE INPUT: initial vector $x_0$, tolerance $\epsilon$, function
$g$, vector $a$.
\STATE OUTPUT: vector $x$, bool $feasible$.
\STATE
\STATE $x := x_0$ \label{alg2_initialization_begin}
\STATE $Q := \varnothing$
\STATE $\xi := x - g(x)$ \label{alg2_initialization_end}
\STATE
\FOR{$i = 1, \ldots, n$} \label{alg_first_iter_begin}
\IF {$[\xi]_i>\epsilon$ }
\STATE $Q := \textrm{Enqueue}(Q,(i,*))$ \label{alg2_enqueue_1}
\ENDIF
\ENDFOR \label{alg_first_iter_end}
\STATE
\WHILE{$Q \neq \varnothing$} \label{alg_main_loop_begin}
    \STATE $(Q,i) := \textrm{Dequeue}(Q)$
    \STATE $[x]_i := [x]_i - [\xi]_i$ \label{alg_change_x}
    \FOR{all $j \in \mathcal{N}(i)$}
        \STATE $[\xi]_j := [x]_j - g_j(x)$ \label{alg_change_xi_2}
        \IF {$[\xi]_j > \epsilon$}
            \STATE $Q := \textrm{Enqueue}(Q,(j,*))$ \label{alg2_enqueue_2}
        \ENDIF
    \ENDFOR
    \STATE $[\xi]_i : =0$ \label{alg_change_xi}
\ENDWHILE \label{alg_main_loop_end}
\STATE
\STATE $feasible := x \geq a$ 
\STATE
\RETURN{$x,feasible$}
\end{algorithmic}
\end{algorithm}

The following proposition characterizes Algorithm~\ref{alg:eps_sol}
and proves its correctness.

\begin{prpstn}
Assume that $x_0 \geq x^+$ and $g(x_0) \geq x_0$, then
Algorithm~\ref{alg:eps_sol} satisfies the following properties:

i) At all times, $x \geq x^+$ and $x \geq g(x)$.

ii) At every evaluation of line~\ref{alg_main_loop_begin}, $x=g(x)+\xi$ and $\xi \geq 0$.

iii) The algorithm terminates in a finite number of steps for any $\epsilon > 0$.

iv) If Problem~\eqref{eqn_prob_class} is feasible, output ``feasible'' is true.

v) If output ``feasible'' is true, then $x$ is an $\epsilon$-feasible
solution of Problem~\eqref{eqn_prob_class}.
\end{prpstn}

\begin{proof}
i) We prove both properties by induction. Note that $x$ is updated
only at line~\ref{alg_change_x} and that line~\ref{alg_change_x} is
equivalent to $[x]_i=g_i(x)$.
For $m \in \mathbb{N}$, let $x(m)$ be the value of $x$ after the $m$-th
evaluation of line~\ref{alg_change_x}.
Note that $x(0) = x_0 \geq x^+$ and that $x$ is changed only at
step~\ref{alg_change_x}. 
Then $[x(m)]_i = g_i(x(m-1)) \geq g_i(x^+) = x^+$,
where we have used the inductive hypothesis $x(m-1) \geq x^+$ and
the fact that $g(x^+)=x^+$ (by Proposition~\ref{prop_max_x}).

Further, note that $g(x(0)) = g(x_0) \geq x_0$ by assumption.
Moreover, $[x(m)]_i = g_i(x(m-1)) = g_i(x(m))$, since $g_i$ does
not depend on $[x]_i$ by assumption and variables $x(m)$, $x(m-1)$
differ only on the $i$-th component.
By the induction hypothesis,
$[x(m)]_i = g_i(x(m-1)) \leq [x(m-1)]_i$ which implies that $x(m) \leq x(m-1)$.
Thus, in view of the monotonicity of $g$ and of the inductive assumption,
for $k \neq i$, $[g(x(m))]_k = g_k(x(m)) \leq g_k(x(m-1)) \leq [x(m-1)]_k = [x(m)]_k$.

ii)  Condition $x=g(x)+\xi$ is satisfied after
evaluating~\ref{alg2_initialization_end}. Moreover, after
evaluating line~\ref{alg_change_xi}, $[x]_i=[g(x)]_i+[\xi]_i$ and all
indices $j$ for which potentially $[x]_j \neq [g(x)]_j+[\xi]_j$ belong
to set $\mathcal{N}(i)$.
For these indices, line~\ref{alg_change_xi_2} re-enforces
$[x]_j = [g(x)]_j+[\xi]_j$.
The fact that $\xi\geq 0$ is a consequence of point i).

iii) At each evaluation of line~\ref{alg_change_x}
 the value of a component of $x$ is decreased by at
least $\epsilon$. If the algorithm did not terminate, at some
iteration we would have that $x
\ngeq x^+$ which is not possible by i).

iv) If Problem~\eqref{eqn_prob_class} is feasible, then  $x^+\geq a$ is its
optimal solution. By point 1), $x \geq x^+ \geq a$ and output
``feasible'' is true. 

v)  When the algorithm terminates, $Q$ is empty, which implies than $\|x-g(x) \|_{\infty} \leq \epsilon$, if
``feasible'' is true, it is also $x\geq a$ and $x$ is an
$\epsilon$-solution.
\end{proof}

\section{Characterization of Problem~\eqref{eqn_prob_class_lin}}
\label{sec_lin_case}

In this section, we consider Problem~\eqref{eqn_prob_class_lin}
and we propose a solution method that exploits its linear structure and
is more efficient than Algorithm~\ref{alg:eps_sol}.
First of all, we show that Problem~\eqref{eqn_prob_class_lin}
belongs to class~\eqref{eqn_prob_class}. To this end, set
\begin{equation}\label{eq:P_def}
P_\ell := I - D_\ell,
\end{equation}
where $I \in \Real^{n \times n}$ is the identity matrix and, for $\ell
\in \mathcal{L}$, $D_\ell \in \Real^{n \times n}$ is a diagonal
matrix that contains the elements of $A_\ell$ on the diagonal.
Note that here and in what follows we assume that all the diagonal
entries of $A_\ell$ are lower than 1.
Indeed, for values larger than or equal to 1 the corresponding
constraints are redundant and can be eliminated.
The proof of the following proposition is in the appendix.

\begin{prpstn}\label{Proposition:reformulation}
Problem~\eqref{eqn_prob_class_lin} can be reformulated as a
problem of class~\eqref{eqn_prob_class}.  Namely, this is achieved by setting
\begin{equation}\label{def:Ab_hat}
\hat A_\ell := {P_\ell}^{-1}(A_\ell - D_\ell), \quad
\hat b_\ell := {P_\ell}^{-1} b_\ell
\end{equation}
and $\hat g(x) = \underset{\ell \in \mathcal{L}}{\glb} \{\hat A_\ell x + \hat
b_\ell\} \wedge U$.
\end{prpstn}

Then we apply the results for Problem~\eqref{eqn_prob_class} to Problem~\eqref{eqn_prob_class_lin}. The following proposition is a corollary of Proposition~\ref{prop_max_x}.
\begin{prpstn}
Problem~\eqref{eqn_prob_class_lin} is feasible and
its optimal solution $x^+$ satisfies the two equations
\begin{equation}\label{eq:problem_intro}
x^+ =\underset{\ell \in \mathcal{L}}{\glb} \left\{ \hat A_\ell x^+ +
  \hat b_\ell \right\} \wedge U\,.
\end{equation}
\begin{equation}\label{eq:fixed_point_ref}
x^+ =\underset{\ell \in \mathcal{L}}{\glb} \left\{ A_\ell x^+ +
  b_\ell \right\} \wedge U\,.
\end{equation}

\end{prpstn}

\begin{proof}
Note that $g(0)= b_\ell \wedge U\geq 0$, which implies
that $\Sigma \neq \varnothing$ and that Problem~\eqref{eqn_prob_class}
is feasible.  Then, by Proposition~\ref{prop_max_x}, its solution
$x^+$ satisfies $x^+=g(x^+)$, which implies~\eqref{eq:problem_intro} and~\eqref{eq:fixed_point_ref}.
\end{proof}

The following result, needed below, can be found, e. g., in \cite{heinonen2005lectures}.

\begin{lmm}\label{lemma:lipschitz}
Let $L \in \Real_+$ and $\{g_i: i \in I\}$, with $I$ set of indices, be a
family of functions
$g_i: \Real^n \rightarrow \Real^n$ such that $\forall x, y \in \Real^n$
\[
\| g_i(x) - g_i(y) \|_\infty \leq L \| x - y \|_\infty.
\]
Then, function $g(x) := \displaystyle\underset{i \in I}{\glb} \{ g_i(x) \}$
also satisfies $\forall x, y \in \Real^n$
\[
\| g(x) - g(y) \|_\infty \leq L \| x - y \|_\infty.
\]

\end{lmm}

The following proposition illustrates that if the infinity norm of all
matrices $A_\ell$ is lower than 1, equation~\eqref{eq:fixed_point_ref} is actually a contraction.

\begin{prpstn}
\label{prop_fix_point_1}
Assume that there exists a real constant $\gamma \in [0,1)$ such that
\begin{equation}
\label{eqn_hyp_fix_point}
\forall \ell \in \mathcal{L},\,\|A_{\ell}\|_\infty < \gamma, 
\end{equation}
then function
\begin{equation}\label{eq:fixed_point_ref2}
\bar g(x) = \underset{\ell \in \mathcal{L}}{\glb} \left\{ A_\ell x +
  b_\ell \right\} \wedge U\,.
\end{equation}
is a contraction in infinity norm, in particular, $\forall x,y \in \Real^n$,
\begin{equation}
\label{eqn_bound_contr1}
\|\bar g(x)-\bar g(y)\|_\infty \leq \gamma \|x-y\|_\infty\,.
\end{equation}
\end{prpstn}

\begin{proof}
Note that, for any $\ell \in \mathcal{L}$, function $h(x)=A_\ell x +
b_\ell$ is a contraction, in fact, for any $x,y \in \Real^n$
\[
\|h(x)-h(y)\|_{\infty} = \|A_\ell(x-y)\|_\infty \leq \gamma \|x-y\|_\infty\,.
\]
Then, the thesis is a consequence of Lemma~\ref{lemma:lipschitz}.
\end{proof}

The following result proves that, under the same assumptions,
also~\eqref{eq:problem_intro} is a contraction. The proof is
in the appendix.

\begin{prpstn}
\label{thm_main}
Assume that~\eqref{eqn_hyp_fix_point} holds
and set
\begin{align}
\label{eqn_choice_hat}
\hat A_\ell = {P_\ell}^{-1}(A_\ell - D_\ell)\quad
\text{ and }\quad
\hat b_\ell = {P_\ell}^{-1} b_\ell,
\end{align}
with $P_\ell$ and $D_\ell$ defined as in~\eqref{eq:P_def}.
Let
\begin{equation}
\label{def:g_hat}
\hat g(x) = \underset{\ell \in \mathcal{L}}{\glb}
\{\hat A_\ell x + \hat b_\ell\} \wedge U,
\end{equation}
then $\hat g$ is a contraction in infinity norm, in particular, $\forall x,y \in \Real^n$,
\begin{equation}
\label{eqn_bound_contr2}
\|\hat g(x) - \hat g(y)\|_\infty \leq \hat \gamma \|x-y\|_\infty\,,
\end{equation}
where
\begin{equation}\label{def:hat_gamma}
\hat\gamma := \max_{\substack{\ell \in \mathcal{L} \\ i \in \mathcal{V}}}
\left\{\frac{\gamma - \left[D_\ell\right]_{ii}}{1 - \left[D_\ell\right]_{ii}}\right\}.
\end{equation}
Moreover, it holds that $\hat\gamma \leq \gamma$.
\end{prpstn}

Hence, in case~\eqref{eqn_hyp_fix_point} is satisfied,
Problem~\eqref{eqn_prob_class_lin} can be solved by 
Algorithm~\ref{alg:fixed_point} using either $g=\bar g$
in~\eqref{eq:fixed_point_ref2} or $g=\hat
g$ in~\eqref{def:g_hat}. As we will show in Section~\ref{sec:convergence_speed_discussion},
the convergence is faster in the second case.

Algorithm~\ref{alg:eps_sol} can be applied to
Problem~\eqref{eqn_prob_class_lin}, being a subclass
of~\eqref{eqn_prob_class}. Anyway, the linear structure of
Problem~\eqref{eqn_prob_class_lin} allows for a more efficient implementation, detailed in
Algorithm~\ref{alg:consensus}.
This algorithm takes as input an initial vector $x_0 \in \Real^n$,
a tolerance $\epsilon$, matrices $A_\ell$ and vectors $b_\ell$, for
$\ell \in \mathcal{L}$, representing function $g$ and the lower bound $a$.
It operates like Algorithm~\ref{alg:eps_sol} but it optimizes the operation
performed in line~\ref{alg_change_xi_2} of Algorithm~\ref{alg:eps_sol}.
Lines from~\ref{alg3_initial_xi_begin} to~\ref{alg3_initial_xi_end}
initialize the error vector $\xi$ and they correspond to
line~\ref{alg2_initialization_end} of Algorithm~\ref{alg:eps_sol}.
Whilst, lines~\ref{alg3_xi_update_begin} from to~\ref{alg3_xi_update_end}
are the equivalent of line~\ref{alg_change_xi_2} of Algorithm~\ref{alg:eps_sol}
in which the special structure of Problem~\eqref{eqn_prob_class_lin}
is exploited in such a way that the updating of the $j$-th component
of vector $\xi$ only involves the evaluation of $L$ scalar products and $L$
scalar sums, with $L = |\mathcal{L}|$, as opposed to (up to) $nL$ scalar products
and $nL$ scalar sums of Algorithm~\ref{alg:eps_sol} applied to
Problem~\eqref{eqn_prob_class_lin}.

\begin{algorithm}[h!]
\caption{Solution algorithm for Problem~\eqref{eqn_prob_class_lin}.}
\label{alg:consensus}
\begin{algorithmic}[1]
\STATE INPUT: initial vector $x_0$, tolerance $\epsilon$,
matrices $A_\ell$, vectors $b_\ell$ for $\ell \in \mathcal{L}$,
vector $a$.
\STATE OUTPUT: vector $x$.

\STATE $x := x_0$
\STATE $Q := \varnothing$
\STATE
\FOR{all $\ell \in \mathcal{L}$} \label{alg3_initial_xi_begin}
    \STATE $\eta_\ell := A_\ell x + b_\ell$
\ENDFOR
\STATE $\xi := x - \underset{\ell \in \mathcal{L}}{\bigwedge} \eta_\ell$ \label{alg3_initial_xi_end}
\STATE
\FOR{all $i \in \mathcal{V}$}
        \IF{($\left[ \xi \right]_i > \epsilon$)}
            \STATE $Q:= \textrm{Enqueue}\left(Q, \left(i, * \right)\right)$ \label{alg3_enqueue_1}
        \ENDIF
\ENDFOR
\STATE
\WHILE{$Q \neq \varnothing$}
    \STATE $(Q, i) := \textrm{Dequeue}(Q)$
    \STATE $[x]_i = [x]_i - [\xi]_i$
    \FOR{all $j \in \mathcal{V} : i \in \mathcal{N}(j)$}
        \FOR{all $\ell \in \mathcal{L}$} \label{alg3_xi_update_begin}
            \STATE $\left[ \eta_\ell \right]_j := \left[ \eta_\ell \right]_j - \left[ A_\ell \right]_{ji} \cdot [\xi]_i$
        \ENDFOR
        \STATE $[\xi]_j := [x]_j - \displaystyle\min_{\ell \in \mathcal{L}} \left[ \eta_\ell \right]_j$ \label{alg3_xi_update_end}
        \IF {$[\xi]_j > \epsilon$}
            \STATE $Q := \textrm{Enqueue}\left(Q, \left(j, * \right)\right)$ \label{alg3_enqueue_2}
        \ENDIF
    \ENDFOR
    \STATE $[\xi]_i = 0$
\ENDWHILE
\STATE
\STATE $feasible := x \geq a$ 
\STATE
\RETURN{$x,feasible$}
\end{algorithmic}
\end{algorithm}

\section{Convergence Speed Discussion}\label{sec:convergence_speed_discussion}

In this section, we will compare the convergence speed of various
methods for solving Problem~\eqref{eqn_prob_class_lin}.
First of all, note that Problem~\eqref{eqn_prob_class_lin} can be
reformulated as the linear problem~\eqref{eqn_prob_class_proglin}.
Hence, it can be solved with any general method for linear problems. As
we will show, the performance of such methods is poor since they do
not exploit the special stucture of Problem~\eqref{eqn_prob_class_proglin}.

\subsection{Fixed point iterations}

In case hypothesis~\eqref{eqn_hyp_fix_point} is satisfied,
as discussed in Section~\ref{sec_lin_case}, 
Problem~\eqref{eqn_prob_class_lin} can be solved by 
Algorithm~\ref{alg:fixed_point} using either $g=\bar g$
in~\eqref{eq:fixed_point_ref2} or $g=\hat
g$ in~\eqref{def:g_hat}. In other words, $x^+$ can be computed with
one of the following iterations:

\begin{equation}
\label{PROB2_noprec}
\begin{cases}
x(k+1) = \bar g(x)=\underset{\ell \in \mathcal{L}}{\glb} \left\{A_\ell
  x(k) + b_\ell \right\} \wedge U\\
x(0) = x_0,
\end{cases}
\end{equation}

\begin{equation}
\label{PROB2}
\begin{cases}
x(k+1) = \hat g(x) =\underset{\ell \in \mathcal{L}}{\glb} \left\{ \hat A_\ell
  x(k) + \hat b_\ell \right\} \wedge U\\
x(0) = x_0,
\end{cases}
\end{equation}
where $x_0 \in \Real^n$ is an arbitrary initial condition and $\hat A_\ell$
and $\hat b_\ell$ are defined as in~\eqref{def:Ab_hat}.

We can compare the convergence rate of iterations~\eqref{PROB2_noprec}
and~\eqref{PROB2}.
The speed of convergence of iteration~\eqref{PROB2_noprec}
can be measured by the convergence rate:
\[
\bar \chi := \max_{\substack{x \in \Real^N \\ x \neq x^\star}} \left\{
  \frac{{\|\bar g(x)-\bar g(x^+)\|}_\infty}{{\|x - x^+\|}_\infty} \right\}.
\]
Similarly, we call $\hat \chi$ the convergence rate of iteration~\eqref{PROB2}.
Note that, by Proposition~\ref{prop_fix_point_1}, $\bar \chi
\leq \gamma$ and, by Proposition~\ref{thm_main},
$\hat \chi \leq \max_{\substack{\ell \in \mathcal{L} \\ i \in \mathcal{V}}}
\left\{\frac{\gamma - \left[D_\ell\right]_{ii}}{1 -
    \left[D_\ell\right]_{ii}}\right\} \leq \gamma$. 
Hence, in general, we have a better upper bound of the convergence
rate of iteration~\eqref{PROB2} than~\eqref{PROB2_noprec}.

Now, let us assume that matrices  $\{A_\ell\}_{\ell \in \mathcal{L}}$
are dominant diagonal, that is, there exists 
$\Delta \in \left[ 0, \frac{1}{2} \right)$ such that, $\forall i \in \{1, \ldots, n\}$,
$\forall \ell \in \mathcal{L}$,
\begin{equation}\label{dominant_diagonal}
[A_\ell]_{ii} \geq (1 - \Delta)\gamma \quad \text{and} \quad
\sum_{\substack{j = 1\\ j \neq i}}^n [A_\ell]_{ij} \leq \Delta \gamma.
\end{equation}
Recall that in the applications discussed in Section~\ref{sec:motivation}
this is attained when $h$ is small enough.
In the following theorem, whose proof is proved in the Appendix, we
state that, if $\Delta$ is small enough, iteration~\eqref{PROB2} has a
faster convergence than iteration~\eqref{PROB}.

\begin{prpstn}
\label{thm:convergence_speed}
Assume that~\eqref{eqn_hyp_fix_point} holds
and let $\Delta \in \left[ 0, \frac{1}{2} \right)$ be such that
matrices $\{A_\ell\}_{\ell \in \mathcal{L}}$ satisfy~\eqref{dominant_diagonal}.
Then, if the starting point $x_0$ is selected in such a way
that $x_0\geq x^+$, then the solutions of both~\eqref{PROB2_noprec}
and~\eqref{PROB2} satisfy $x(k) \geq x_0$, $\forall k \in
\mathbb{N}$. Moreover, if
\[
\Delta \in \left[ 0, \frac{\sqrt{1 - \gamma} - (1 - \gamma)}{\gamma} \right)\,,
\]
then for any $x \geq x^+$
\begin{equation}\label{convergence}
{\|\hat g(x) - x^+\|}_\infty < {\| \bar g(x)- x^+\|}_\infty.
\end{equation}
\end{prpstn}

\subsection{Speed of Algorithm~\ref{alg:consensus} and priority queue policy}
\label{section:priority_queue}

As we will see in the numerical
experiments section, Algorithm~\ref{alg:consensus} solves
Problem~\eqref{eqn_prob_class_lin} more efficiently
than iterations~\eqref{PROB2_noprec} and~\eqref{PROB2}.

As we already mentioned in the previous section, the order in which we
update the values of the nodes in the priority queue does not affect the
convergence of the algorithm but impacts heavily on its convergence speed.
We implemented four different queue policies, detailed in the following.

\subsubsection{Node variation}
The priority associated to an index $i$ is given by
the opposite of the absolute value of the variation of $[x]_i$ in its last update.
In this case, in lines~\ref{alg2_enqueue_1},~\ref{alg2_enqueue_2} of
Algorithm~\ref{alg:eps_sol} and lines~\ref{alg3_enqueue_1},~\ref{alg3_enqueue_2}
of Algorithm~\ref{alg:consensus}, symbol $*$ is replaced by the opposite of
the corresponding component of $\xi$ of the node added to the queue
(see Table~\ref{tab:code}).
This can be considered a ``greedy'' policy, in fact we update first the
components of the solution $[x]_i$ associated to
a larger variation $[\xi]_i$, in order to have a faster
convergence of the current solution $x$ to $x^+$.

\subsubsection{Node values}
The priority associated to an index $i$ in the priority queue is given by $[x]_i$.
In this case, in lines~\ref{alg2_enqueue_1},~\ref{alg2_enqueue_2} of
Algorithm~\ref{alg:eps_sol} and lines~\ref{alg3_enqueue_1},~\ref{alg3_enqueue_2}
of Algorithm~\ref{alg:consensus}, symbol $*$ is replaced by the opposite of the value
of the node added to the queue (see Table~\ref{tab:code}).
The rationale of this policy is the observation that, in Problem~\eqref{eqn_prob_class_proglin}, components of $x$
with lower values are more likely to appear in active constraints.
This policy mimics Dijkstra's
algorithm, in fact the indexes associated to the solution components
with lower values are processed first.

\subsubsection{FIFO e LIFO policies}
The two remaining policies implement respectively the First In First Out (FIFO)
policy, (i.e., a stack) and the Last In First Out (LIFO) policy (i.e.,
a queue). Namely, in case of FIFO, the nodes are updated in the order in which
they are inserted in the queue. In case of LIFO, they are updated in
reverse order.

In order to formally implement these two policies in a priority queue, we need to introduce a counter $k$ initialized to $0$ and
incremented every time a node is added to the priority queue.
In lines~\ref{alg2_enqueue_1},~\ref{alg2_enqueue_2} of Algorithm~\ref{alg:eps_sol}
and lines~\ref{alg3_enqueue_1},~\ref{alg3_enqueue_2} of Algorithm~\ref{alg:consensus},
symbol $*$ is replaced by $k$ in case we want to implement a LIFO policy and by
$-k$ for implementing a FIFO policy (see Table~\ref{tab:code}).
These steps are required to formally represent these two policies in
Algorithm~\ref{alg:consensus}. As said, these two policies can be more
simply implemened with an unordered queue (for FIFO policy) or a stack (for LIFO
policy). The rationale of this two policies is to avoid the overhead
of managing a priority queue. In fact, inserting an entry into a
priority queue of $n$ elements has a time-cost of $O(\log
n)$, while the same operation on an unordered queue or a stack has a
cost of $O(1)$. Note that, with these policies, we increase the efficiency in the management of the set
of the indexes that have to be updated at the expense of a possible
less efficient update policy.

\begin{table}[ht]
\centering
\begin{tabular}{|c|l|l|}
\hline
Policy &
Alg.\ref{alg:eps_sol} line~\ref{alg2_enqueue_1},
Alg.\ref{alg:consensus} line~\ref{alg3_enqueue_1} &
Alg.\ref{alg:eps_sol} line~\ref{alg2_enqueue_2},
Alg.\ref{alg:consensus} line~\ref{alg3_enqueue_2} \\
\hline
Variation & $Q:= \textrm{Enqueue}\left(Q, \left(i, -[\xi]_i \right)\right)$ &
$Q:= \textrm{Enqueue}\left(Q, \left(j, -[\xi]_j \right)\right)$ \\
\hline
Value & $Q:= \textrm{Enqueue}\left(Q, \left(i, [x]_i \right)\right)$ &
$Q:= \textrm{Enqueue}\left(Q, \left(j, [x]_j \right)\right)$ \\
\hline
FIFO & $Q:= \textrm{Enqueue}\left(Q, \left(i, k \right)\right)$; $k := k + 1$ &
$Q:= \textrm{Enqueue}\left(Q, \left(j, k \right)\right)$; $k := k + 1$ \\
\hline
LIFO & $Q:= \textrm{Enqueue}\left(Q, \left(i, -k \right)\right)$; $k := k + 1$ &
$Q:= \textrm{Enqueue}\left(Q, \left(j, -k \right)\right)$; $k := k + 1$ \\
\hline
\end{tabular}
\caption{Possible priority queue policies.}
\label{tab:code}
\end{table}

\subsection{Numerical Experiments}\label{subsec:simulations}

In this section, we test Algorithm~\ref{alg:consensus} on randomly
generated problems of class~\eqref{eqn_prob_class_lin}.
We carried out two sets of tests. In the first one, we compared the solution time
of Algorithm~\ref{alg:consensus} with different priority queue
policies with a commercial solver for linear problems (Gurobi).
In the second class of tests, we compared the number of scalar
multiplications executed by Algorithm~\ref{alg:consensus} (with different
priority queue policies) with the ones
required by the fixed point iteration~\eqref{PROB2_noprec}.

\subsubsection{Random problems generation}
The following procedure allows generating a random problem of
class~\eqref{eqn_prob_class_lin} with $n$ variables. The procedure
takes the following input parameters:

\begin{itemize}
\item $U \in \Real^+$: an upper bound for the problem solution,
\item $M_A \in \Real^+$: maximum value for entries of $A_1,\ldots,A_L$,
\item $M_b \in \Real^+$: maximum value for entries of $b_1,\ldots,b_L$,
\item $G_1,\ldots, G_L$: graphs with $n$ nodes.
\end{itemize} 

A problem of class~\eqref{eqn_prob_class_lin} is then obtained with
the following operations, for $i=1,\ldots,L$:

\begin{itemize}
\item Set $D_i$ as the adiacency matrix of graph $G_i$,
\item define $A_i$ as the matrix obtained from $D_i$ by replacing each
  nonzero entry of $D_i$ with a random number generated from
  a uniform distribution in interval $[0,M_A]$,
\item define $b_i \in \Real^n$ so that each entry is a random number generated from a uniform
  distribution in interval $[0,M_b]$.
\end{itemize}

Graphs $G_1,\ldots,G_L$ are obtained from standard classes of random
graphs, namely:
\begin{itemize}
\item the Barab\'asi-Albert model~\cite{BARABASI1999}, characterized
by a scale-free degree distribution,
\item the Newman-Watts-Strogatz model~\cite{NEWMAN1999}, that originates
  graphs with small-world properties,
\item  the Holm and Kim algorithm~\cite{HOLME2002}, that produces
scale-free graphs with high clustering.
\end{itemize}
In our tests, we used the software NetworkX~\cite{NetworkX} to
generate the random graphs.

\subsubsection{Test 1: solution time}
We considered random instances of Problem~\eqref{eqn_prob_class_lin}
obtained with the following parameters: $U = 10^{5}$, $M_A = 0.5$,
$M_b = 1$, $L = 4$, using random graphs with a varying  number of nodes
obtained with the following models.
\begin{itemize}
\item  The Barab\'asi-Albert model (see~\cite{BARABASI1999} for more
  details), in which each new node is connected to
5 existing nodes.
\item The Watts-Strogatz model (see~\cite{NEWMAN1999}), in which each node is connected to its 2 nearest
neighbors and with shortcuts created with a probability of 3 divided by the number
of nodes in the graph.
\item The Holm and Kim algorithm (see~\cite{HOLME2002}), in which 4 random edges are added for each
new node and with a probability of $0.25$ of adding an extra random
edge generating a triangle.
\end{itemize}

Figures~\ref{fig:barabasi_albert_graph_time},~\ref{fig:newman_watts_strogatz_graph_time}
and~\ref{fig:powerlaw_cluster_graph_time} compare the solution times obtained
with Algoritm~\ref{alg:consensus} (using different queue policies) to those obtained
with Gurobi. The figures refer to random graphs generated with Barab\'asi-Albert model, Watts-Strogatz model
and Holm and Kim algorithm, respectively. For each figure, the
horizontal axis
represents the number
of variables (that are logarithmically spaced) and the vertical-axis represents
the solution times (also logarithmically spaced), obtained as the average of 5 tests.
For each graph type, the policies based on FIFO and node variation
appear to be the best performing ones. In particular,  for problems obtained from the
Barabasi-Albert model (Figure~\ref{fig:barabasi_albert_graph_time})
and Holm and Kim algorithm
(Figure~\ref{fig:powerlaw_cluster_graph_time}), the solution
time obtained with these two policies are more than three orders of
magnitude lower than Gurobi. Moreover, the solution
time with FIFO policy is more than one order
of magniture lower than Gurobi for problems obtained from
Watts-Strogatz model (Figure~\ref{fig:newman_watts_strogatz_graph_time}).
Note that, in every figure, Gurobi solution times are almost constant for
small numbers of variables. A possible explanation could be that Gurobi performs
some dimension-independent operations which, at small dimensions, are the most
time-consuming ones.
Note also that, in Figures~\ref{fig:barabasi_albert_graph_time} 
and~\ref{fig:powerlaw_cluster_graph_time}, the solution times for node
value and LIFO policies are missing starting from a certain number of
variables. This is due to
excessively high computational times, however, the first collected data points are enough
for drawing conclusions on the performances of these policies which, as the number of
variables grows, perform far worse than Gurobi.

\noindent
\begin{tikzpicture}

\begin{axis}[%
width=.8\textwidth,
height=.5\textwidth,
at={(1.011in,0.642in)},
scale only axis,
xmode=log,
xmin=10,
xmax=11144,
xminorticks=true,
xlabel style={font=\color{white!15!black}},
xlabel={Number of variables},
ymode=log,
ymin=0.0001,
ymax=1000,
yminorticks=true,
ylabel style={font=\color{white!15!black}},
ylabel={Solution time},
axis background/.style={fill=white},
xmajorgrids,
xminorgrids,
ymajorgrids,
yminorgrids,
legend style={at={(0.99,0.01)}, anchor=south east, legend cell align=left, align=left, draw=white!15!black}
]
\addplot [color=mycolor1, line width=1.0pt, mark=asterisk, mark options={solid, mycolor1}]
  table[row sep=crcr]{%
10	0.0005951548\\
13	0.00035344\\
17	0.000499044\\
21	0.0005792682\\
27	0.000787925\\
35	0.0092325724\\
45	0.005667969\\
58	0.0017644296\\
74	0.0022005946\\
95	0.0032154846\\
123	0.0055618398\\
157	0.0070056918\\
202	0.012333153\\
260	0.0159228652\\
334	0.0156618468\\
429	0.0265329208\\
551	0.035822053\\
708	0.0551215674\\
910	0.0573981438\\
1169	0.0684891284\\
1501	0.084946479\\
1929	0.134060193\\
2478	0.1310034238\\
3184	0.169810344\\
4090	0.209420571\\
5255	0.2568849752\\
6752	0.3660565208\\
8674	0.4712324662\\
11144	0.5907037372\\
14318	0\\
18395	0\\
23634	0\\
30364	0\\
39010	0\\
50119	0\\
};
\addlegendentry{Variation}

\addplot [color=mycolor2, line width=1.0pt, mark=o, mark options={solid, mycolor2}]
  table[row sep=crcr]{%
10	0.0006611324\\
13	0.000167685\\
17	0.0002435054\\
21	0.000359978\\
27	0.0004558798\\
35	0.0015596832\\
45	0.0013026902\\
58	0.0015923736\\
74	0.0026346076\\
95	0.0029254636\\
123	0.0030095482\\
157	0.0045819746\\
202	0.0044292348\\
260	0.0150607736\\
334	0.0079751382\\
429	0.014015092\\
551	0.0198774706\\
708	0.0352434538\\
910	0.0331614908\\
1169	0.0521926136\\
1501	0.0583760694\\
1929	0.1083283926\\
2478	0.0887241906\\
3184	0.1045216054\\
4090	0.123326712\\
5255	0.1338982558\\
6752	0.1638463362\\
8674	0.2036605504\\
11144	0.2729580392\\
14318	0\\
18395	0\\
23634	0\\
30364	0\\
39010	0\\
50119	0\\
};
\addlegendentry{FIFO}

\addplot [color=mycolor3, line width=1.0pt, mark=square, mark options={solid, mycolor3}]
  table[row sep=crcr]{%
10	0.00314563\\
13	0.0122389824\\
17	0.028481783\\
21	0.0540521318\\
27	0.1000488006\\
35	0.386978246\\
45	0.8433000552\\
58	1.3572409162\\
74	3.6587333724\\
95	0\\
123	0\\
157	0\\
202	0\\
260	0\\
334	0\\
429	0\\
551	0\\
708	0\\
910	0\\
1169	0\\
1501	0\\
1929	0\\
2478	0\\
3184	0\\
4090	0\\
5255	0\\
6752	0\\
8674	0\\
11144	0\\
14318	0\\
18395	0\\
23634	0\\
30364	0\\
39010	0\\
50119	0\\
};
\addlegendentry{Value}

\addplot [color=mycolor4, line width=1.0pt, mark=pentagon, mark options={solid, mycolor4}]
  table[row sep=crcr]{%
10	0.022491162\\
13	0.0735789134\\
17	0.176859916\\
21	0.463932564\\
27	0.8965218696\\
35	3.0192616228\\
45	0\\
58	0\\
74	0\\
95	0\\
123	0\\
157	0\\
202	0\\
260	0\\
334	0\\
429	0\\
551	0\\
708	0\\
910	0\\
1169	0\\
1501	0\\
1929	0\\
2478	0\\
3184	0\\
4090	0\\
5255	0\\
6752	0\\
8674	0\\
11144	0\\
14318	0\\
18395	0\\
23634	0\\
30364	0\\
39010	0\\
50119	0\\
};
\addlegendentry{LIFO}

\addplot [color=mycolor5, line width=1.0pt, mark=diamond, mark options={solid, mycolor5}]
  table[row sep=crcr]{%
10	0.3459368862\\
13	0.266633962\\
17	0.264651968\\
21	0.2609900964\\
27	0.269805837\\
35	0.2952675002\\
45	0.3570352102\\
58	0.3352853552\\
74	0.3410980454\\
95	0.3438401718\\
123	0.2922790924\\
157	0.3539103666\\
202	0.3833662614\\
260	0.3534752564\\
334	0.3928540948\\
429	0.5392693906\\
551	0.4639342778\\
708	1.1983676908\\
910	1.8217313006\\
1169	2.0217317532\\
1501	3.0225821664\\
1929	5.942938245\\
2478	8.7950893672\\
3184	16.3822081754\\
4090	28.8719912166\\
5255	62.1615212322\\
6752	126.0592345666\\
8674	250.490089081\\
11144	509.2188303676\\
14318	0\\
18395	0\\
23634	0\\
30364	0\\
39010	0\\
50119	0\\
};
\addlegendentry{Gurobi}

\end{axis}
\end{tikzpicture}

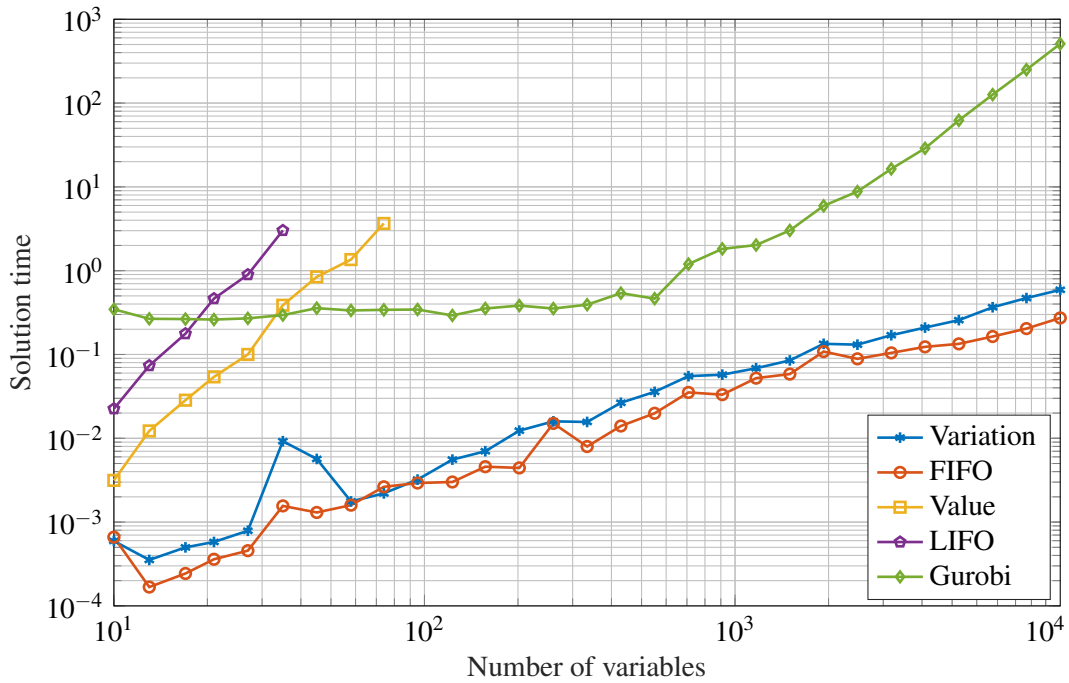
\captionof{figure}{Solution times for graphs with growing number of nodes generated with Barab\'asi-Albert model.}
\label{fig:barabasi_albert_graph_time}

\begin{tikzpicture}

\begin{axis}[%
width=.8\textwidth,
height=.5\textwidth,
at={(1.011in,0.642in)},
scale only axis,
xmode=log,
xmin=10,
xmax=50119,
xminorticks=true,
xlabel style={font=\color{white!15!black}},
xlabel={Number of variables},
ymode=log,
ymin=0.0001,
ymax=100,
yminorticks=true,
ylabel style={font=\color{white!15!black}},
ylabel={Solution time},
axis background/.style={fill=white},
xmajorgrids,
xminorgrids,
ymajorgrids,
yminorgrids,
legend style={at={(0.99,0.01)}, anchor=south east, legend cell align=left, align=left, draw=white!15!black}
]
\addplot [color=mycolor1, line width=1.0pt, mark=asterisk, mark options={solid, mycolor1}]
  table[row sep=crcr]{%
10	0.0437058822\\
13	0.0003105642\\
17	0.0003349432\\
21	0.0003880072\\
27	0.0004336068\\
35	0.0005247948\\
45	0.000590372\\
58	0.0007067506\\
74	0.0008802644\\
95	0.0010568416\\
123	0.001462452\\
157	0.0018199536\\
202	0.0022388226\\
260	0.0038207278\\
334	0.0056529644\\
429	0.0071710784\\
551	0.0108080108\\
708	0.0162415596\\
910	0.020296989\\
1169	0.0286340496\\
1501	0.036942233\\
1929	0.0423397638\\
2478	0.056434152\\
3184	0.0654424698\\
4090	0.121155158\\
5255	0.120423036\\
6752	0.1400534612\\
8674	0.178852733\\
11144	0.2649161826\\
14318	0.4155184038\\
18395	0.562368231\\
23634	0.6613739906\\
30364	0.9988348148\\
39010	1.2374534102\\
50119	1.631090228\\
};
\addlegendentry{Variation}

\addplot [color=mycolor2, line width=1.0pt, mark=o, mark options={solid, mycolor2}]
  table[row sep=crcr]{%
10	0.000501501\\
13	0.0001129486\\
17	0.0001205568\\
21	0.0001454514\\
27	0.0001625716\\
35	0.0001890596\\
45	0.0002200992\\
58	0.0002302208\\
74	0.0002439548\\
95	0.0005181754\\
123	0.0005229488\\
157	0.0007790132\\
202	0.0013731628\\
260	0.0023038752\\
334	0.0022321066\\
429	0.0025945732\\
551	0.0024802142\\
708	0.0038862076\\
910	0.005670732\\
1169	0.0081429976\\
1501	0.008239282\\
1929	0.009971903\\
2478	0.0122597952\\
3184	0.0170984232\\
4090	0.0358771618\\
5255	0.04051714\\
6752	0.0593637738\\
8674	0.0605770912\\
11144	0.0714349828\\
14318	0.1042273148\\
18395	0.1109716336\\
23634	0.1069118182\\
30364	0.1207863376\\
39010	0.143382946\\
50119	0.1822079336\\
};
\addlegendentry{FIFO}

\addplot [color=mycolor3, line width=1.0pt, mark=square, mark options={solid, mycolor3}]
  table[row sep=crcr]{%
10	0.0030631324\\
13	0.003665062\\
17	0.005082224\\
21	0.0077198636\\
27	0.0100503912\\
35	0.0132994676\\
45	0.0181489316\\
58	0.0252437954\\
74	0.0237352274\\
95	0.0302054782\\
123	0.035608264\\
157	0.052395607\\
202	0.0653032904\\
260	0.0807155248\\
334	0.1048264554\\
429	0.1261190088\\
551	0.1544885632\\
708	0.2104607114\\
910	0.2499445\\
1169	0.3309349264\\
1501	0.4111294766\\
1929	0.546758185\\
2478	0.6694439334\\
3184	0.8522266348\\
4090	1.4145447466\\
5255	1.4698170538\\
6752	1.8312934464\\
8674	2.3711939058\\
11144	3.1252997074\\
14318	4.283921471\\
18395	5.329291155\\
23634	6.524120078\\
30364	9.2146832476\\
39010	12.3157065618\\
50119	13.9710775174\\
};
\addlegendentry{Value}

\addplot [color=mycolor4, line width=1.0pt, mark=pentagon, mark options={solid, mycolor4}]
  table[row sep=crcr]{%
10	0.0018031918\\
13	0.0019955598\\
17	0.0021590202\\
21	0.0026767832\\
27	0.0034084286\\
35	0.0038898614\\
45	0.0041752668\\
58	0.0079182316\\
74	0.007292178\\
95	0.0113519546\\
123	0.0138527408\\
157	0.0188655924\\
202	0.0216657636\\
260	0.0264216568\\
334	0.0332667152\\
429	0.0407806204\\
551	0.0484925108\\
708	0.0602456228\\
910	0.0744995958\\
1169	0.091407883\\
1501	0.1183014852\\
1929	0.1474717314\\
2478	0.188861047\\
3184	0.240157985\\
4090	0.4588354656\\
5255	0.4368420284\\
6752	0.51562392\\
8674	0.6565395882\\
11144	0.9455767692\\
14318	1.1497900564\\
18395	1.4300264812\\
23634	1.8082655934\\
30364	2.8185721652\\
39010	3.4079966926\\
50119	3.8371527132\\
};
\addlegendentry{LIFO}

\addplot [color=mycolor5, line width=1.0pt, mark=diamond, mark options={solid, mycolor5}]
  table[row sep=crcr]{%
10	0.323342488\\
13	0.3281140364\\
17	0.3381316846\\
21	0.3451169632\\
27	0.3367729224\\
35	0.3357352846\\
45	0.324017452\\
58	0.3408094084\\
74	0.3164435598\\
95	0.3062755636\\
123	0.245843998\\
157	0.2501585924\\
202	0.2417714116\\
260	0.2656039764\\
334	0.2419686076\\
429	0.2642374488\\
551	0.2663696784\\
708	0.2597513194\\
910	0.2750517558\\
1169	0.27871696\\
1501	0.287340832\\
1929	0.3033246612\\
2478	0.5301388348\\
3184	0.620885309\\
4090	0.9149669274\\
5255	0.9121104986\\
6752	1.1056631104\\
8674	1.2273643368\\
11144	1.6046847444\\
14318	1.6613798818\\
18395	1.5180824458\\
23634	1.9858505602\\
30364	2.5996911352\\
39010	3.2992850724\\
50119	3.3848775542\\
};
\addlegendentry{Gurobi}

\end{axis}
\end{tikzpicture}

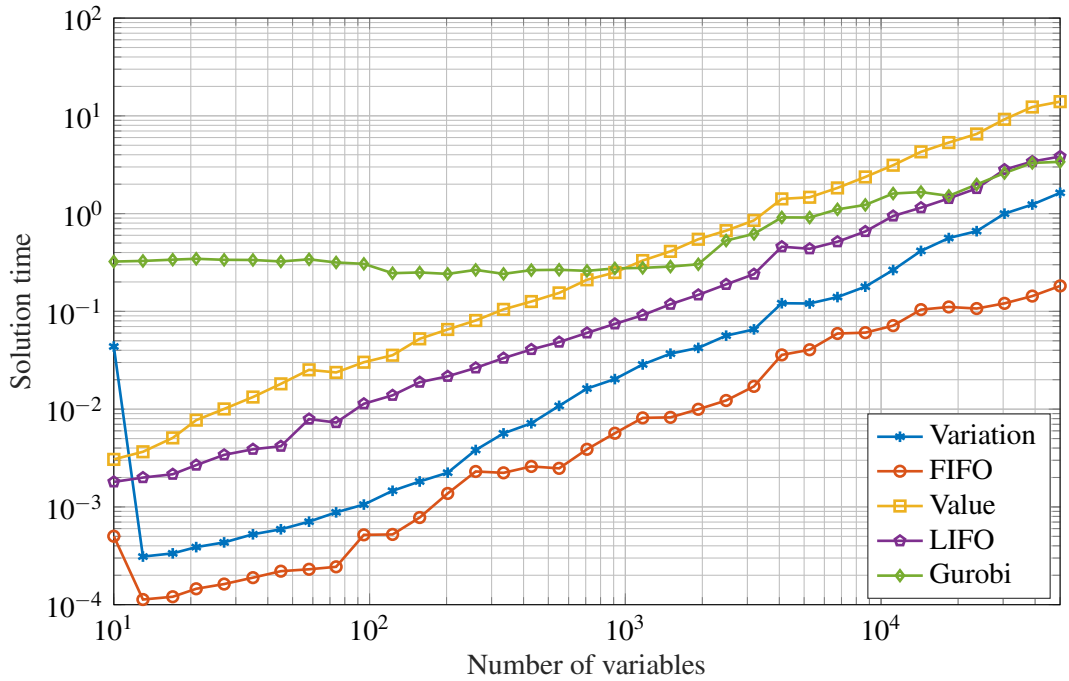
\captionof{figure}{Solution times for graphs with growing number of nodes generated with Newman-Watts-Strogatz model.}
\label{fig:newman_watts_strogatz_graph_time}

\begin{tikzpicture}

\begin{axis}[%
width=.8\textwidth,
height=.5\textwidth,
at={(1.011in,0.642in)},
scale only axis,
xmode=log,
xmin=10,
xmax=11144,
xminorticks=true,
xlabel style={font=\color{white!15!black}},
xlabel={Number of variables},
ymode=log,
ymin=0.0001,
ymax=1000,
yminorticks=true,
ylabel style={font=\color{white!15!black}},
ylabel={Solution time},
axis background/.style={fill=white},
xmajorgrids,
xminorgrids,
ymajorgrids,
yminorgrids,
legend style={at={(0.99,0.01)}, anchor=south east, legend cell align=left, align=left, draw=white!15!black}
]
\addplot [color=mycolor1, line width=1.0pt, mark=asterisk, mark options={solid, mycolor1}]
  table[row sep=crcr]{%
10	0.0536987564\\
13	0.0010019214\\
17	0.0004300098\\
21	0.0005706234\\
27	0.0007647136\\
35	0.00107908\\
45	0.0013368298\\
58	0.0022689742\\
74	0.0028972518\\
95	0.0032404698\\
123	0.0053247938\\
157	0.0061977208\\
202	0.008244831\\
260	0.0157692264\\
334	0.0162276656\\
429	0.0268331206\\
551	0.037888087\\
708	0.0580116484\\
910	0.0770917402\\
1169	0.0743282476\\
1501	0.0720739864\\
1929	0.1072650794\\
2478	0.1465275914\\
3184	0.1879428798\\
4090	0.2150498624\\
5255	0.2758929592\\
6752	0.341359006\\
8674	0.5186461124\\
11144	0.7728565042\\
14318	0\\
18395	0\\
23634	0\\
30364	0\\
39010	0\\
50119	0\\
};
\addlegendentry{Variation}

\addplot [color=mycolor2, line width=1.0pt, mark=o, mark options={solid, mycolor2}]
  table[row sep=crcr]{%
10	0.0029447472\\
13	0.0001863554\\
17	0.0003917964\\
21	0.000379023\\
27	0.0005147618\\
35	0.0005317338\\
45	0.0008674904\\
58	0.0031011384\\
74	0.0018998266\\
95	0.0038669768\\
123	0.0025136482\\
157	0.0040747384\\
202	0.0052481556\\
260	0.0089187658\\
334	0.0120926002\\
429	0.0198741096\\
551	0.0223323406\\
708	0.0369560926\\
910	0.0553750908\\
1169	0.0413713452\\
1501	0.0470457412\\
1929	0.0830445404\\
2478	0.0968584854\\
3184	0.1115110838\\
4090	0.1177151406\\
5255	0.1228187552\\
6752	0.1477737006\\
8674	0.2115624136\\
11144	0.2633009292\\
14318	0\\
18395	0\\
23634	0\\
30364	0\\
39010	0\\
50119	0\\
};
\addlegendentry{FIFO}

\addplot [color=mycolor3, line width=1.0pt, mark=square, mark options={solid, mycolor3}]
  table[row sep=crcr]{%
10	0.0055850004\\
13	0.0115391894\\
17	0.0251107596\\
21	0.0578950462\\
27	0.1027943962\\
35	0.2856070582\\
45	0.4858989502\\
58	1.0325888364\\
74	2.3098624084\\
95	4.2030332666\\
123	0\\
157	0\\
202	0\\
260	0\\
334	0\\
429	0\\
551	0\\
708	0\\
910	0\\
1169	0\\
1501	0\\
1929	0\\
2478	0\\
3184	0\\
4090	0\\
5255	0\\
6752	0\\
8674	0\\
11144	0\\
14318	0\\
18395	0\\
23634	0\\
30364	0\\
39010	0\\
50119	0\\
};
\addlegendentry{Value}

\addplot [color=mycolor4, line width=1.0pt, mark=pentagon, mark options={solid, mycolor4}]
  table[row sep=crcr]{%
10	0.0220423906\\
13	0.0454583432\\
17	0.1392907014\\
21	0.2452004284\\
27	0.6623545588\\
35	2.001916001\\
45	3.6018202132\\
58	0\\
74	0\\
95	0\\
123	0\\
157	0\\
202	0\\
260	0\\
334	0\\
429	0\\
551	0\\
708	0\\
910	0\\
1169	0\\
1501	0\\
1929	0\\
2478	0\\
3184	0\\
4090	0\\
5255	0\\
6752	0\\
8674	0\\
11144	0\\
14318	0\\
18395	0\\
23634	0\\
30364	0\\
39010	0\\
50119	0\\
};
\addlegendentry{LIFO}

\addplot [color=mycolor5, line width=1.0pt, mark=diamond, mark options={solid, mycolor5}]
  table[row sep=crcr]{%
10	0.3474432906\\
13	0.2888704704\\
17	0.3037165692\\
21	0.299908297\\
27	0.3093113462\\
35	0.3001741816\\
45	0.3209836864\\
58	0.3036790408\\
74	0.3182152444\\
95	0.3046952044\\
123	0.3147784406\\
157	0.3247152874\\
202	0.3111099764\\
260	0.3381673284\\
334	0.3644460596\\
429	0.4187763042\\
551	0.4201275096\\
708	0.5188806938\\
910	1.699291754\\
1169	2.2449774678\\
1501	3.0619865764\\
1929	4.6667603954\\
2478	7.5373405438\\
3184	12.3516520966\\
4090	22.143164653\\
5255	43.2589357954\\
6752	87.53416575\\
8674	183.744954342\\
11144	393.8327459402\\
14318	0\\
18395	0\\
23634	0\\
30364	0\\
39010	0\\
50119	0\\
};
\addlegendentry{Gurobi}

\end{axis}
\end{tikzpicture}

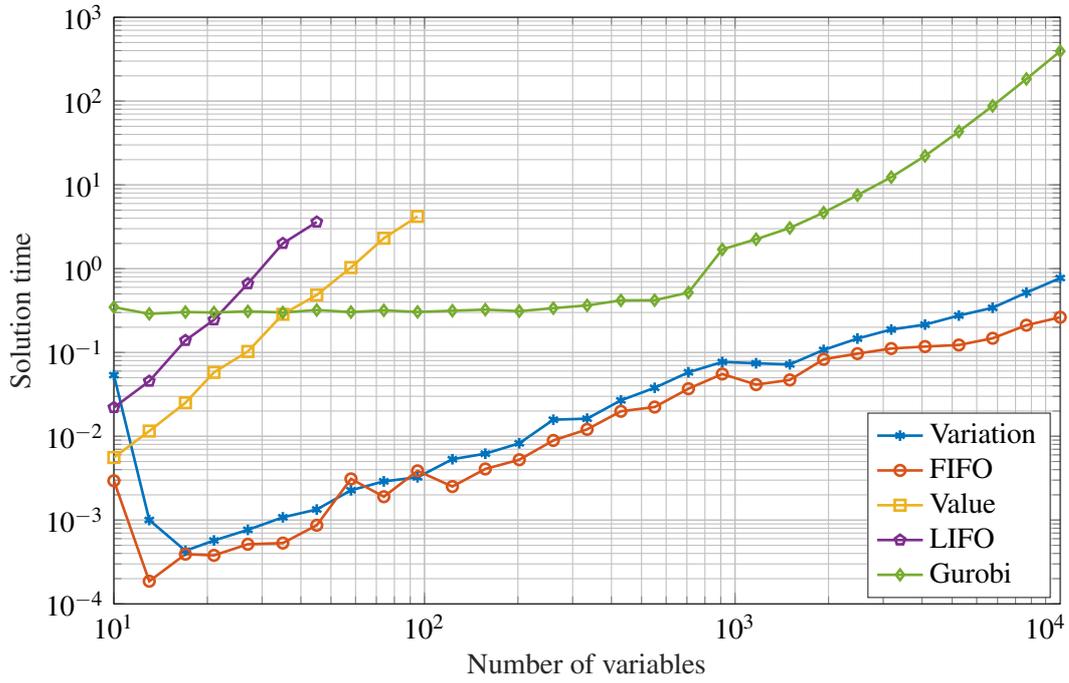
\captionof{figure}{Solution times for graphs with growing number of nodes generated with Holm and Kim algorithm.}
\label{fig:powerlaw_cluster_graph_time}

\begin{tikzpicture}

\begin{axis}[%
width=.8\textwidth,
height=.5\textwidth,
at={(1.011in,0.642in)},
scale only axis,
x dir=reverse,
xmode=log,
xmin=1e-10,
xmax=0.1,
xminorticks=true,
xlabel style={font=\color{white!15!black}},
xlabel={Tolerance},
ymode=log,
ymin=10000,
ymax=10000000,
yminorticks=true,
ylabel style={font=\color{white!15!black}},
ylabel={Multiplications},
axis background/.style={fill=white},
xmajorgrids,
xminorgrids,
ymajorgrids,
yminorgrids,
legend style={at={(0.99,0.2)}, anchor=south east, legend cell align=left, align=left, draw=white!15!black}
]
\addplot [color=mycolor1, line width=1.0pt, mark=asterisk, mark options={solid, mycolor1}]
  table[row sep=crcr]{%
0.1	24832\\
0.01	25733\\
0.001	26606.2\\
0.0001	27500.6\\
1e-05	28388\\
1e-06	29292.2\\
1e-07	30182.2\\
1e-08	31059.6\\
1e-09	31964.8\\
1e-10	32856.4\\
};
\addlegendentry{Variation}

\addplot [color=mycolor2, line width=1.0pt, mark=o, mark options={solid, mycolor2}]
  table[row sep=crcr]{%
0.1	25092\\
0.01	26076.6\\
0.001	27074.2\\
0.0001	28079.6\\
1e-05	29083.8\\
1e-06	30074.6\\
1e-07	31087\\
1e-08	32087.8\\
1e-09	33084.4\\
1e-10	34086\\
};
\addlegendentry{FIFO}

\addplot [color=mycolor3, line width=1.0pt, mark=square, mark options={solid, mycolor3}]
  table[row sep=crcr]{%
0.1	1245061.2\\
0.01	4310056\\
0.001	0\\
0.0001	0\\
1e-05	0\\
1e-06	0\\
1e-07	0\\
1e-08	0\\
1e-09	0\\
1e-10	0\\
};
\addlegendentry{Value}

\addplot [color=mycolor4, line width=1.0pt, mark=pentagon, mark options={solid, mycolor4}]
  table[row sep=crcr]{%
0.1	50777.6\\
0.01	70298.8\\
0.001	99497\\
0.0001	140810.2\\
1e-05	216532\\
1e-06	359016.6\\
1e-07	584156.2\\
1e-08	1016941.2\\
1e-09	1989840.2\\
1e-10	3891423.4\\
};
\addlegendentry{LIFO}

\addplot [color=mycolor5, line width=1.0pt, mark=diamond, mark options={solid, mycolor5}]
  table[row sep=crcr]{%
0.1	396000\\
0.01	475200\\
0.001	534600\\
0.0001	594000\\
1e-05	673200\\
1e-06	732600\\
1e-07	792000\\
1e-08	871200\\
1e-09	930600\\
1e-10	990000\\
};
\addlegendentry{Fixed point}

\end{axis}
\end{tikzpicture}

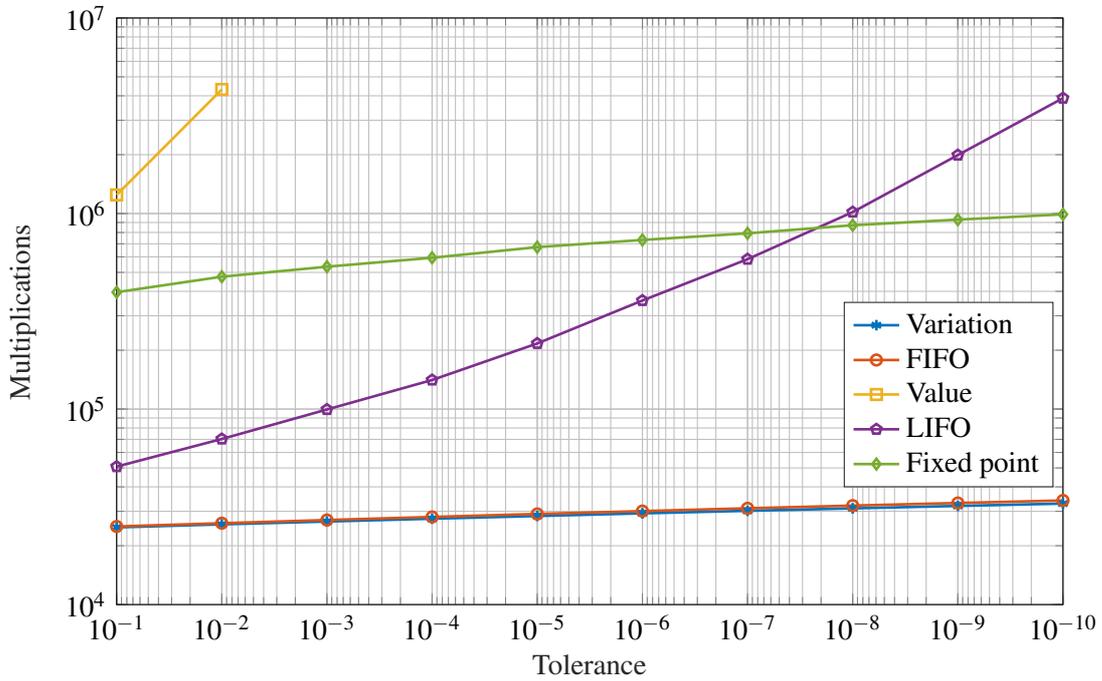
\captionof{figure}{Scalar multiplications for different tolerances on a graph generated with Barab\'asi-Albert model.}
\label{fig:barabasi_albert_graph_multiplications}

\begin{tikzpicture}

\begin{axis}[%
width=.8\textwidth,
height=.5\textwidth,
at={(1.011in,0.642in)},
scale only axis,
x dir=reverse,
xmode=log,
xmin=1e-10,
xmax=0.1,
xminorticks=true,
xlabel style={font=\color{white!15!black}},
xlabel={Tolerance},
ymode=log,
ymin=8890.4,
ymax=1000000,
yminorticks=true,
ylabel style={font=\color{white!15!black}},
ylabel={Multiplications},
axis background/.style={fill=white},
xmajorgrids,
xminorgrids,
ymajorgrids,
yminorgrids,
legend style={at={(0.99,0.2)}, anchor=south east, legend cell align=left, align=left, draw=white!15!black}
]
\addplot [color=mycolor1, line width=1.0pt, mark=asterisk, mark options={solid, mycolor1}]
  table[row sep=crcr]{%
0.1	8890.4\\
0.01	9707.6\\
0.001	10534.6\\
0.0001	11372.6\\
1e-05	12204\\
1e-06	13037.6\\
1e-07	13872.6\\
1e-08	14702\\
1e-09	15537\\
1e-10	16372.2\\
};
\addlegendentry{Variation}

\addplot [color=mycolor2, line width=1.0pt, mark=o, mark options={solid, mycolor2}]
  table[row sep=crcr]{%
0.1	8956.2\\
0.01	9957.6\\
0.001	10953.8\\
0.0001	11948.2\\
1e-05	12930.8\\
1e-06	13912.8\\
1e-07	14885\\
1e-08	15860.4\\
1e-09	16831.8\\
1e-10	17805.2\\
};
\addlegendentry{FIFO}

\addplot [color=mycolor3, line width=1.0pt, mark=square, mark options={solid, mycolor3}]
  table[row sep=crcr]{%
0.1	31863.4\\
0.01	58205.4\\
0.001	111309.6\\
0.0001	220091.4\\
1e-05	442670.4\\
1e-06	895020.4\\
1e-07	0\\
1e-08	0\\
1e-09	0\\
1e-10	0\\
};
\addlegendentry{Value}

\addplot [color=mycolor4, line width=1.0pt, mark=pentagon, mark options={solid, mycolor4}]
  table[row sep=crcr]{%
0.1	20940.6\\
0.01	29482.8\\
0.001	41511.2\\
0.0001	60431.6\\
1e-05	88596\\
1e-06	132722.6\\
1e-07	212574\\
1e-08	311660.4\\
1e-09	437125.2\\
1e-10	761744.2\\
};
\addlegendentry{LIFO}

\addplot [color=mycolor5, line width=1.0pt, mark=diamond, mark options={solid, mycolor5}]
  table[row sep=crcr]{%
0.1	80000\\
0.01	96000\\
0.001	108000\\
0.0001	120000\\
1e-05	136000\\
1e-06	148000\\
1e-07	160000\\
1e-08	176000\\
1e-09	188000\\
1e-10	200000\\
};
\addlegendentry{Fixed point}

\end{axis}
\end{tikzpicture}

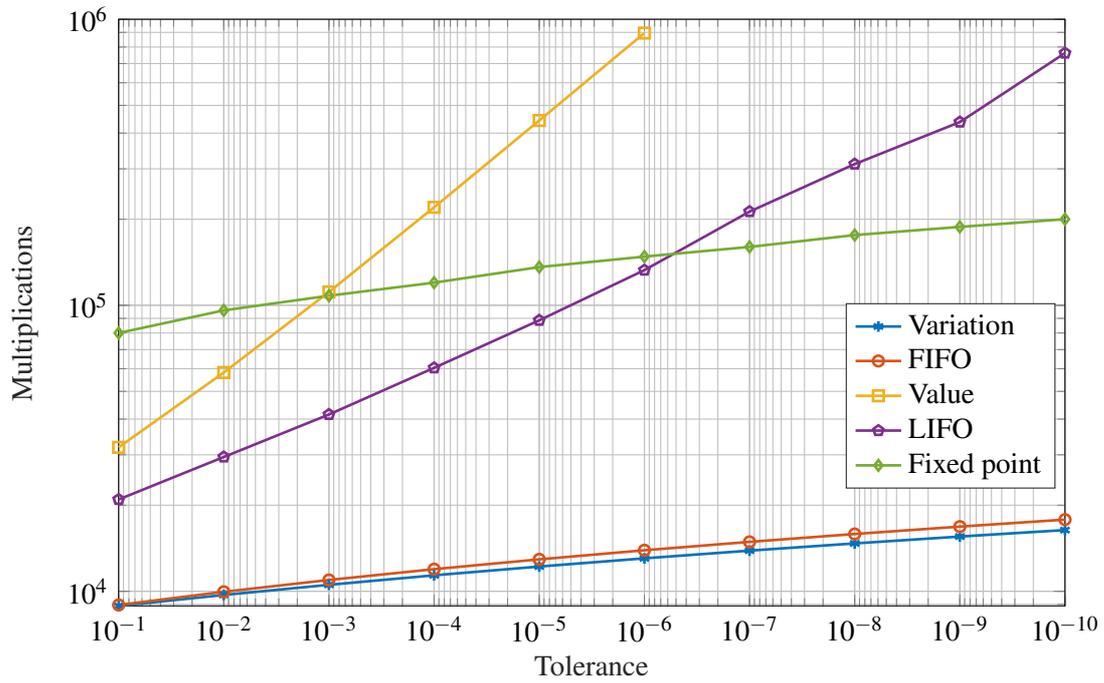
\captionof{figure}{Scalar multiplications for different tolerances on a graph generated with Newman-Watts-Strogatz model.}
\label{fig:newman_watts_strogatz_graph_multiplications}

\begin{tikzpicture}

\begin{axis}[%
width=.8\textwidth,
height=.5\textwidth,
at={(1.011in,0.642in)},
scale only axis,
x dir=reverse,
xmode=log,
xmin=1e-10,
xmax=0.1,
xminorticks=true,
xlabel style={font=\color{white!15!black}},
xlabel={Tolerance},
ymode=log,
ymin=10000,
ymax=10000000,
yminorticks=true,
ylabel style={font=\color{white!15!black}},
ylabel={Multiplications},
axis background/.style={fill=white},
xmajorgrids,
xminorgrids,
ymajorgrids,
yminorgrids,
legend style={at={(0.99,0.2)}, anchor=south east, legend cell align=left, align=left, draw=white!15!black}
]
\addplot [color=mycolor1, line width=1.0pt, mark=asterisk, mark options={solid, mycolor1}]
  table[row sep=crcr]{%
0.1	20777.4\\
0.01	21665.4\\
0.001	22535.4\\
0.0001	23428.2\\
1e-05	24309\\
1e-06	25193\\
1e-07	26106.6\\
1e-08	26975.8\\
1e-09	27870.6\\
1e-10	28750.8\\
};
\addlegendentry{Variation}

\addplot [color=mycolor2, line width=1.0pt, mark=o, mark options={solid, mycolor2}]
  table[row sep=crcr]{%
0.1	21006\\
0.01	22028.4\\
0.001	23052.8\\
0.0001	24063.6\\
1e-05	25099.4\\
1e-06	26124.4\\
1e-07	27148\\
1e-08	28157\\
1e-09	29174.4\\
1e-10	30221.8\\
};
\addlegendentry{FIFO}

\addplot [color=mycolor3, line width=1.0pt, mark=square, mark options={solid, mycolor3}]
  table[row sep=crcr]{%
0.1	883130.8\\
0.01	3037159.8\\
0.001	0\\
0.0001	0\\
1e-05	0\\
1e-06	0\\
1e-07	0\\
1e-08	0\\
1e-09	0\\
1e-10	0\\
};
\addlegendentry{Value}

\addplot [color=mycolor4, line width=1.0pt, mark=pentagon, mark options={solid, mycolor4}]
  table[row sep=crcr]{%
0.1	43825.6\\
0.01	60735.8\\
0.001	86885.8\\
0.0001	123018.4\\
1e-05	187654\\
1e-06	307452.8\\
1e-07	525227.4\\
1e-08	874317\\
1e-09	1754279.8\\
1e-10	3582942\\
};
\addlegendentry{LIFO}

\addplot [color=mycolor5, line width=1.0pt, mark=diamond, mark options={solid, mycolor5}]
  table[row sep=crcr]{%
0.1	317440\\
0.01	380928\\
0.001	428544\\
0.0001	476160\\
1e-05	539648\\
1e-06	587264\\
1e-07	634880\\
1e-08	698368\\
1e-09	745984\\
1e-10	793600\\
};
\addlegendentry{Fixed point}

\end{axis}
\end{tikzpicture}

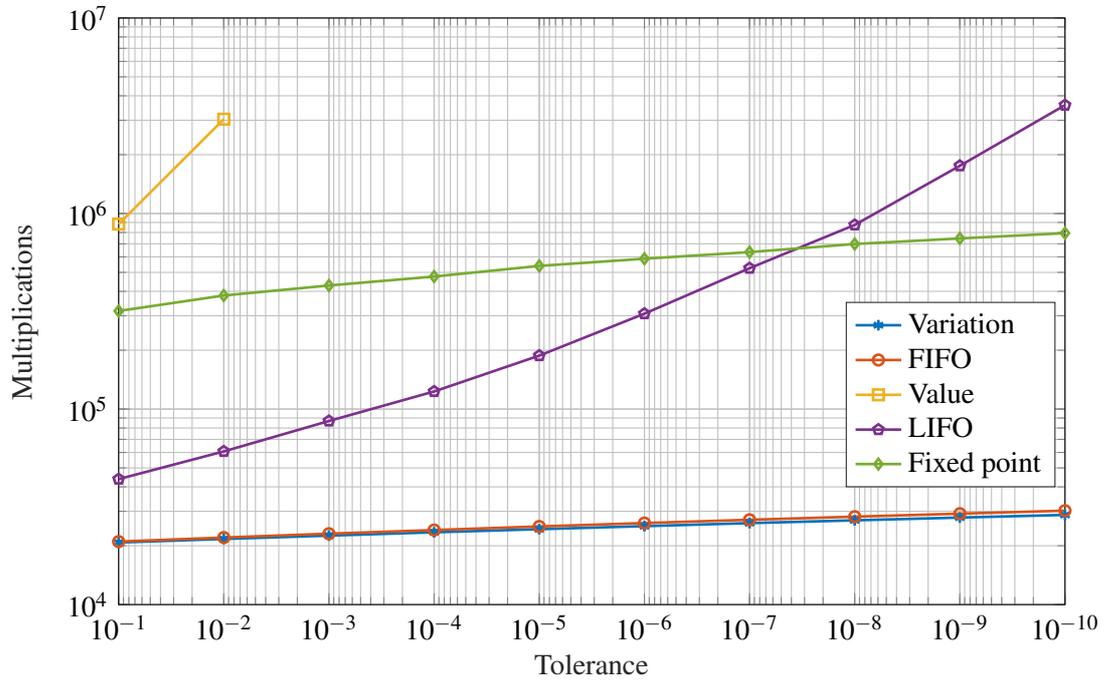
\captionof{figure}{Scalar multiplications for different tolerances on a graph generated with Holm and Kim algorithm.}
\label{fig:powerlaw_cluster_graph_multiplications}

\subsubsection{Test 2: number of operations}
We considered three instances of Problem~\eqref{eqn_prob_class_lin},
obtained from the three classes of random graphs considered in the
previous tests, with the same parameters and with 500 nodes. 
For each instance, we considered 10 logarithmically spaced values of
tolerance $\epsilon$ between $10^{-1}$ and $10^{-10}$.
We solved each problem with the following methods:
\begin{itemize}
\item the preconditioned fixed point iteration~\eqref{PROB2},
\item  Algorithm~\ref{alg:consensus} with FIFO, LIFO, node value and
  node variation policies.
\end{itemize}

The results are reported in Figures~\ref{fig:barabasi_albert_graph_multiplications},~\ref{fig:newman_watts_strogatz_graph_multiplications}
and~\ref{fig:powerlaw_cluster_graph_multiplications}.
These figures show that the number of product operations required with
node variation policy is much lower (of one order of magnitude) than those
required by the fixed point iteration~\eqref{PROB2_noprec}.
The iteration based on FIFO, even though slightly less performing than the
the node variation policy, also gives comparable results to it.
Observe that, even though the iteration based on node variation requires (slightly) less
scalar multiplications than the one based on FIFO, its solution times are worse
than those obtained with the FIFO policy, since the management of the priority
queue based on node variation is computationally more demanding than a 
First-In-First-Out data structure.
The iteration based on nodes value provides poor performances even with high
tolerances.
Also, the iteration based on LIFO gives poor computational results,
underperforming the fixed point iteration~\eqref{PROB2_noprec} for tolerances
smaller than $10^{-7}$, in Figures~\ref{fig:barabasi_albert_graph_multiplications}
and~\ref{fig:powerlaw_cluster_graph_multiplications}, and smaller than $10^{-6}$,
in Figure~\ref{fig:newman_watts_strogatz_graph_multiplications}.
Note that, in Figures~\ref{fig:barabasi_albert_graph_multiplications},~\ref{fig:newman_watts_strogatz_graph_multiplications}
and~\ref{fig:powerlaw_cluster_graph_multiplications}, below a certain value of
the tolerance, the numbers of scalar multiplications for the priority queue based
on node value are missing due to excessively high computational times.
However, the first collected data points are enough for drawing conclusions on
the performances of this policy.
\newline\newline\noindent
As a concluding remark, we observe that all the experiments confirm our previous claim about the relevance of the ordering in the priority queue. While convergence is guaranteed for all the orderings we tested, speed of convergence and number of scalar multiplications turn out to be rather different between them. 
In what follows we give a tentative explanation of such different performances.
The good performance of the node variation policy can be explained with the fact that such 
policy guarantees a quick reduction of the variables values. The LIFO and value orderings seem to update a small subset of variables before proceeding to update also the other variables. This is particularly evident in the case of the value policy, where  only variables with small values are initially updated. The FIFO ordering guarantees a more uniform propagation of the updates, thus avoiding stagnation into small portions of the feasible region.

\section*{Appendix: Proofs of the Main Results}\label{sec:proof}

\subsection{Proof of Proposition~\ref{Proposition:reformulation}}
\begin{proof}
Given $A \in \Real^{n \times n}$ let us define, for $i=1,\ldots,n$ the
sum of the elements of row $i$
\begin{equation}
\label{eqn_def_s}
s_i(A) := \sum_{j = 1}^n {[A]}_{ij}.
\end{equation}

\noindent
Note that, for  any $\ell \in \mathcal{L}$, matrix $P_\ell$ defined
in~\eqref{eq:P_def} is  positive
diagonal since, by assumption, all elements of $D_\ell$ are
less than $1$.
One can rewrite the inequality of Problem~\eqref{eqn_prob_class_lin} as
\begin{align*}
&\underline 0 \leq \underset{\ell \in \mathcal{L}}{\glb} \{ A_\ell x - (D_\ell - D_\ell + I)x+ b_\ell \} \\
\Leftrightarrow\ &
\underline 0 \leq \underset{\ell \in \mathcal{L}}{\glb} \{(A_\ell - D_\ell) x - (I - D_\ell)x + b_\ell \} \\
\Leftrightarrow\ &
\underline 0 \leq \underset{\ell \in \mathcal{L}}{\glb} \{(I - D_\ell)^{-1}(A_\ell - D_\ell) x - x + (I - D_\ell)^{-1}b_\ell \} \\
\Leftrightarrow\ &
x \leq \underset{\ell \in \mathcal{L}}{\glb} \{(I - D_\ell)^{-1}(A_\ell - D_\ell) x + (I - D_\ell)^{-1}b_\ell \} \\
\Leftrightarrow\ &
x \leq \underset{\ell \in \mathcal{L}}{\glb} \{{P_\ell}^{-1}(A_\ell - D_\ell) x + {P_\ell}^{-1}b_\ell \}
\end{align*}

Then, set $\hat A_\ell := {P_\ell}^{-1}(A_\ell - D_\ell)$ and
$\hat b_\ell := {P_\ell}^{-1} b_\ell$
and $\hat g(x)=\underset{\ell \in \mathcal{L}}{\bigwedge} \{\hat g_\ell(x)\} \wedge U$,
where, for $\ell \in \mathcal{L}$,
\begin{equation}\label{def:hat_g_ell}
\hat g_\ell(x) := \hat A_\ell x + \hat b_\ell.
\end{equation}
Note that $\hat g$ is monotonic (since all entries of $\hat A_\ell$
are nonnegative) and for $i=1,\ldots,n$, $\left[\hat g \right]_i$ is independent on
$x_i$ (since the diagonal entries of $\hat A_\ell$ are null).
Note also that $\hat b_\ell$ is nonnegative.
Hence, Problem~\eqref{eqn_prob_class_lin} takes on the form of
Problem~\eqref{eqn_prob_class}.
\end{proof}

\subsection{Proof of Proposition~\ref{thm_main}}
Given $P_\ell$ as in~\eqref{eq:P_def} for $i \in \mathcal{V}$
and $\ell \in \mathcal{L}$ we have that
\begin{align*}
s_i(\hat A_\ell) \leq \frac{\gamma - \left[D_\ell\right]_{ii}}{\left[P_\ell\right]_{ii}},
\end{align*}
where $s_i$ is defined in~\eqref{eqn_def_s} and $\hat A_\ell$ is defined as in~\eqref{def:Ab_hat}.
Let us note that
\begin{equation}\label{S_def}
\max_{\substack{\ell \in \mathcal{L} \\ i \in \mathcal{V}}}
\left\{ s_i(\hat A_\ell) \right\} \leq
\max_{\substack{\ell \in \mathcal{L} \\ i \in \mathcal{V}}}
\left\{ \frac{\gamma - \left[D_\ell\right]_{ii}}{\left[P_\ell\right]_{ii}} \right\}
= \max_{\substack{\ell \in \mathcal{L} \\ i \in \mathcal{V}}}
\left\{\frac{\gamma - \left[D_\ell\right]_{ii}}{1 - \left[D_\ell\right]_{ii}}\right\}
= \hat\gamma,
\end{equation}
where $\hat \gamma$ is defined as in~\eqref{def:hat_gamma}.
Note that the term on the left-hand side is the maximum of $s_i(A)$
for all possible $i \in \mathcal{V}$ and for all possible matrices
$A \in \Real^{n \times n}$ which can be obtained by all possible
combinations of the rows of matrices $A_\ell$, with $\ell \in \mathcal{L}$.
We prove that $\hat\gamma \leq \gamma$, under the given assuptions.
Indeed, it is immediate to see that function
\begin{equation}\label{eq:Sd}
S(d) := \frac{\gamma - d}{1 - d}
\end{equation}
is monotone decreasing for any $d \in [0, \gamma]$.
We remark that, for any
$\ell \in \mathcal{L}$, $\left\| \hat A_\ell \right\|_\infty \leq \hat \gamma$.
Now, for any $x \in \Real^n$, let us define $\hat g_U(x) := U$, while
for any $\ell \in \mathcal{L}$, $\hat g_\ell(x)$ is defined as in~\eqref{def:hat_g_ell}.
It is immediate to see that $\forall x, y \in \Real^n$,
$\left\| \hat g_i(x) - \hat g_i(y) \right\|_\infty \leq \hat \gamma \| x- y \|_\infty$,
for any $i \in \mathcal{L} \cup \{U\}$.
Then, by Lemma~\ref{lemma:lipschitz} we have that, for
$\hat g(x) = \displaystyle\bigwedge_{k \in \mathcal{L} \cup \{U\}} \hat g_k(x)$,
it holds that $\forall x, y \in \Real^n$,
$\left\| \hat g(x) - \hat g(y) \right\|_\infty \leq \hat \gamma \| x - y \|_\infty$,
that is, $\hat g$ is a contraction.

\subsection{Proof of Proposition~\ref{thm:convergence_speed}}
We first remark that $x_0 \geq x^+$ implies $x_k \geq x^+$
and $\bar g(x_k) \geq x^+$ for any $k$, where $\bar g$ is defined
as in~\eqref{eq:fixed_point_ref2}.
Then, we provide a lower bound for $\left\| \bar g(x_k) - \bar g(x^+) \right\|_\infty$.
Let $\bar A \in \Real^{n \times n}_+$ and $\bar b \in \Real^n_+$
be such that $\bar A x_k + \bar b = g(x_k)$.
Note that $\bar A$ is obtained by a combination of the rows of matrices
$A_\ell$, with $\ell \in \mathcal{L}$.
In other words, for each $i \in \{1, \ldots, n\}$,
$\left[ \bar A \right]_{i*} = \left[ A_{\ell_i} \right]_{i*}$ for some
$\ell_i \in \mathcal{L}$.
Then, in view of $x_k \geq x^+$, $x^+ \leq \bar A x^+ + \bar b$
and $\bar A \geq 0$,
\begin{align*}
\left\| \bar g(x_k) - \bar g(x^+) \right\|_\infty
=& \left\| \bar A x_k + \bar b - x^+ \right\|_\infty \geq
\left\| \bar A x_k + \bar b - (\bar A x^+ + \bar b) \right\|_\infty
= \left\| \bar A(x_k - x^+) \right\|_\infty \geq \\
\geq& \left\| \diag(\bar A) (x_k - x^+) \right\|_\infty
\geq (1 - \Delta)\gamma \left\| x_k - x^+ \right\|_\infty,
\end{align*}
where the last inequality follows from \eqref{dominant_diagonal}.
Then, the result follows by observing that
\begin{gather*}
\frac{\Delta \gamma}{1 - (1 - \Delta)\gamma} < (1 - \Delta)\gamma\
\Leftrightarrow\
\Delta \gamma < (1 - \Delta)\gamma - (1 - \Delta)^2\gamma^2\
\Leftrightarrow \\
\Leftrightarrow\
\gamma^2\Delta^2 + 2(1 - \gamma)\gamma\Delta - (1 - \gamma)\gamma < 0
\Leftrightarrow\
\Delta \in \left[\left. 0, \frac{\sqrt{1 - \gamma} - (1 - \gamma)}{\gamma} \right.\right).
\end{gather*}

\bibliographystyle{abbrv}
\bibliography{biblio}

\end{document}